\documentclass[journal]{IEEEtran}
\usepackage[cmex10]{amsmath}
\usepackage{amsfonts,amssymb,amsthm,mathtools}
\usepackage{upgreek}
\usepackage{algorithmic}
\usepackage{url}
\usepackage{hyperref}
\usepackage{derivative}
\usepackage{physics2}
\usephysicsmodule{ab, ab.legacy, op.legacy, xmat}
\usepackage{tikz}
\usetikzlibrary{calc}
\theoremstyle{definition}
\newtheorem{definition}{Definition}
\newtheorem{example}[definition]{Example}
\newtheorem{remark}[definition]{Remark}
\theoremstyle{plain}
\newtheorem{theorem}{Theorem}
\newtheorem{proposition}[theorem]{Proposition}
\newtheorem{lemma}[theorem]{Lemma}

\usepackage{pifont}% 
\newcommand{\xmark}{\ding{55}}%
\newcommand{\Incorrect}{\xmark}

\begin{document}
\title{Foundation of Calculating Normalized Maximum Likelihood for Continuous Probability Models}

\author{Atsushi Suzuki, Kota Fukuzawa, and Kenji Yamanishi,~\IEEEmembership{Senior Member,~IEEE,}% <-this % stops a space
\thanks{Atsushi Suzuki is with the Department
of Informatics, King's College London, UK e-mail: atsushi.suzuki.rd@gmail.com.}% <-this % stops a space
\thanks{Kota Fukuzawa and Kenji Yamanishi are Graduate School of Information Science and Technology, The University of Tokyo.}% <-this % stops a space
% \thanks{Manuscript received April 19, 2005; revised September 17, 2014.}%
}

\markboth{Journal of \LaTeX\ Class Files,~Vol.~13, No.~9, September~2014}%
{Shell \MakeLowercase{\textit{et al.}}: Bare Demo of IEEEtran.cls for Journals}



\maketitle

\begin{abstract}
The normalized maximum likelihood (NML) code length is widely used as a model selection criterion based on the minimum description length principle, where the model with the shortest NML code length is selected.
A common method to calculate the NML code length is to use the sum (for a discrete model) or integral (for a continuous model) of a function defined by the distribution of the maximum likelihood estimator.
While this method has been proven to correctly calculate the NML code length of discrete models, no proof has been provided for continuous cases. 
Consequently, it has remained unclear whether the method can accurately calculate the NML code length of continuous models.
In this paper, we solve this problem affirmatively, proving that the method is also correct for continuous cases.
Remarkably, completing the proof for continuous cases is non-trivial in that it cannot be achieved by merely replacing the sums in discrete cases with integrals, as the decomposition trick applied to sums in the discrete model case proof is not applicable to integrals in the continuous model case proof.
To overcome this, we introduce a novel decomposition approach based on the coarea formula from geometric measure theory, 
which is essential to establishing our proof for continuous cases.

\end{abstract}


% \fi
\IEEEpeerreviewmaketitle

\newcommand{\NewTerm}[1]{\textbf{\textit{#1}}}
\newcommand{\Emph}[1]{\textbf{#1}}
\newcommand{\Real}{\mathbb{R}}
\newcommand{\argmin}{\mathop{\mathrm{arg\,min}}}
\newcommand{\argmax}{\mathop{\mathrm{arg\,max}}}
\newcommand{\esssup}{\mathop{\mathrm{ess\,sup}}}
\newcommand{\essinf}{\mathop{\mathrm{ess\,inf}}}
\newcommand{\Integer}{\mathbb{Z}}
\newcommand{\DefEq}{\coloneqq}
\newcommand{\ML}{\mathrm{ML}}
\newcommand{\NML}{\mathrm{NML}}
\newcommand{\DataRV}{X}
\newcommand{\DataRVVec}{\boldsymbol{\DataRV}}
\newcommand{\DataVal}{x}
\newcommand{\DataValVec}{\boldsymbol{\DataVal}}
\newcommand{\DataSpace}{\mathcal{X}}
\newcommand{\DataDim}{D}
\newcommand{\Param}{\theta}
\newcommand{\ParamVec}{\boldsymbol{\theta}}
\newcommand{\ParamMap}{\mathrm{\uptheta}}
\newcommand{\ParamMapVec}{\boldsymbol{\ParamMap}}
\newcommand{\ParamMLE}{\hat{\mathrm{\uptheta}}}
\newcommand{\ParamMLEVec}{\boldsymbol{\ParamMLE}}
\newcommand{\ParamSpace}{\varTheta}
\newcommand{\ParamDim}{K}
\newcommand{\IDatum}{n}
\newcommand{\NData}{N}
\newcommand{\ProbabilityMeasure}{\mu}
\newcommand{\BaseMeasure}{\lambda}
\newcommand{\PDF}{p}
\newcommand{\PMF}{P}
\newcommand{\Probability}{\mathop{\mathrm{Pr}}}
\newcommand{\CodeLen}{l}
\newcommand{\Func}{\mathrm{f}}
\newcommand{\FuncI}{\mathrm{f}}
\newcommand{\FuncII}{\mathrm{g}}
\newcommand{\FuncIII}{\mathrm{h}}
\newcommand{\FuncVec}{\boldsymbol{\Func}}
\newcommand{\FuncVecI}{\boldsymbol{\FuncI}}
\newcommand{\FuncVecII}{\boldsymbol{\FuncII}}
\newcommand{\FuncVecIII}{\boldsymbol{\FuncIII}}
\newcommand{\Set}{A}
\newcommand{\SetI}{A}
\newcommand{\SetII}{B}
\newcommand{\SetIII}{C}
\newcommand{\Lebesgue}{\mathcal{L}}
\newcommand{\Hausdorff}{\mathcal{H}}
\newcommand{\SufficientSymbol}{\mathrm{s}}
\newcommand{\SufficientStat}{\mathrm{s}}
\newcommand{\SufficientStatVec}{\boldsymbol{\SufficientStat}}
\newcommand{\SufficientVal}{s}
\newcommand{\SufficientValVec}{\boldsymbol{\SufficientVal}}
\newcommand{\Complexity}{\textsc{comp}}
\newcommand{\ExponentialComplexity}{\textsc{Comp}}
\newcommand{\ParametricComplexity}{\mathrm{MC}}
\newcommand{\Luckiness}{v}
\newcommand{\Diam}{\mathop{\mathrm{diam}}}
\newcommand{\TagIncorrect}{\stepcounter{equation}\tag{\theequation-\Incorrect}}

\section{Introduction}
\IEEEPARstart{I}{n} data science, learning is to find the regularity underlying a given finite-length sequence of observed data \cite{Grunwald:2007}. 
In the context of model selection with a probabilistic framework, such as statistics, machine learning, or information theory, a regularity is formulated as a probabilistic model. 
In this context, the goal is to select the most suitable one among several candidate probabilistic models. 
Here, a probabilistic model refers to a set of probability distributions. In the following, we will always parametrize the probability distributions belonging to a probability model.

In general, given a sequence of observed data of finite length, a classical principle for selecting one among several possible explanations is the minimum description length (MDL) criterion. 
This principle states that the explanation that most concisely describes the observed data sequence should be selected. 
To apply this principle in reality, we need to quantify the brevity of the description. 
In the context of probabilistic model selection, we need the quantitatively determine the length required for the probabilistic model to describe the observed data sequence. 
The length is formulated based on information theory as the \NewTerm{normalized maximum likelihood (NML) code length} \cite{Shtarkov:1987}, also called \NewTerm{stochastic complexity} \cite{Rissanen:1989} \cite{Rissanen:1996}. 
Therefore, probabilistic model selection based on the MDL criterion refers to the procedure to select a probabilistic model with the shortest NML code length for an observed data sequence.
There have been many model selection applications based on the NML code length (e.g., \cite{tabus2003classification} \cite{Myung+:2006} \cite{Roos+:2008} \cite{schmidt2010estimating} \cite{giurcuaneanu2011variable} \cite{zhang2012model} \cite{Hirai&Yamanishi:2013:Efficient} 
\cite{staniczenko2014selecting} \cite{Ito:2016} \cite{Suzuki:2016} \cite{yamanishi2019decomposed} \cite{kellen2020selecting} \cite{yamanishi2023detecting}).

Specifically, if we have a set of probabilistic models $\ab \{\ProbabilityMeasure^{(j)}_{\ParamSpace^{(j)}}\}_{j \in J}$, then the MDL-principle using the NML code length is to select the probabilistic model index $j \in J$ that minimizes $\CodeLen_{\NML} \ab[\ProbabilityMeasure^{(j)}_{\ParamSpace^{(j)}}] \ab (\ParamVec)$, where $\CodeLen_{\NML} \ab[\ProbabilityMeasure^{(j)}_{\ParamSpace^{(j)}}]$ is the NML code length function of the probabilistic model $\ProbabilityMeasure^{(j)}_{\ParamSpace^{(j)}}$. 
Here, if a probabilistic model $\ProbabilityMeasure_{\ParamSpace} = \ab\{\ProbabilityMeasure_{\ParamVec}\}_{\ParamVec \in \ParamSpace}$ with parameter $\ParamVec$ is given, the NML code length $\CodeLen_\NML \ab[\ProbabilityMeasure_{\ParamSpace}] \ab(\DataValVec)$ for a data (sequence) $\DataValVec \in \DataSpace$ is defined as follows:
\begin{equation}
\CodeLen_\NML \ab[\ProbabilityMeasure_{\ParamSpace}] \ab(\DataValVec)
\DefEq
\CodeLen_\ML \ab[\ProbabilityMeasure_{\ParamSpace}] \ab(\DataValVec)
+
\Complexity \ab(\ProbabilityMeasure_{\ParamSpace}).
\end{equation}
Here, $\Complexity \ab(\ProbabilityMeasure_{\ParamSpace})$ is the value independent of $\DataValVec$ called the \NewTerm{model complexity (MC)} (e.g., \cite{grunwald2019minimum}) or \NewTerm{parametric complexity} (e.g., \cite{Grunwald:2007}) of the probability model $\ProbabilityMeasure_{\ParamSpace}$ defined later, and the log-max-likelihood $\CodeLen_\ML \ab[\ProbabilityMeasure_{\ParamSpace}] \ab(\DataValVec)$ is defined by
\begin{equation}
\CodeLen_\ML \ab[\ProbabilityMeasure_{\ParamSpace}] \ab(\DataValVec)
\DefEq
\begin{cases}
- \log_{b} \PMF \ab[\ProbabilityMeasure_{\ParamMLEVec \ab(\DataValVec)}] \ab(\DataValVec) & \quad \textbf{(discrete case)},
\\    
- \log_{b} \PDF \ab[\ProbabilityMeasure_{\ParamMLEVec \ab(\DataValVec)}] \ab(\DataValVec) & \quad \textbf{(continuous case)},
\\
\end{cases}
\end{equation}
where $\PMF \ab[\ProbabilityMeasure]$ and $\PDF \ab[\ProbabilityMeasure]$ are the probability mass function (PMF) and probability density function (PDF) of the distribution $\ProbabilityMeasure$, respectively, the $b$ is a fixed positive number ($b=2$ and $b=e$ are often chosen for discrete and continuous cases, respectively), and the map $\ParamMLEVec$ is the maximum likelihood estimator defined by
\begin{equation}
\ParamMLEVec \ab(\DataValVec)
\DefEq
\begin{cases}
\displaystyle \argmax_{\ParamVec \in \ParamSpace} \PMF \ab[\ProbabilityMeasure_{\ParamVec}] \ab(\DataValVec) & \quad \textbf{(discrete case)},
\\    
\displaystyle \argmax_{\ParamVec \in \ParamSpace} \PDF \ab[\ProbabilityMeasure_{\ParamVec}] \ab(\DataValVec) & \quad \textbf{(continuous case)}.
\\
\end{cases}
\end{equation}
Intuitively speaking, $\CodeLen_\ML \ab[\ProbabilityMeasure_{\ParamSpace}] \ab(\DataValVec)$ indicates the minimum code length that would be required to describe the data $\DataValVec$ when we knew the best distribution $\ProbabilityMeasure_{\ParamMLEVec \ab(\DataValVec)}$ in advance.
In reality, since we cannot know the best distribution in advance, we cannot achieve as short a code length as what $\CodeLen_\ML \ab[\ProbabilityMeasure_{\ParamSpace}]$ provides.
Indeed, any code length function $\CodeLen$ must satisfy the Kraft-McMillan inequality \cite{Kraft:1949} \cite{McMillan:1956}
\begin{equation}
\label{eqn:Kraft}
\begin{split}
\sum_{\DataValVec \in \DataSpace} b^{-\CodeLen (\DataValVec)} & \le 1 \quad \textbf{(discrete case)}, \\
\int_{\DataSpace} b^{-\CodeLen (\DataValVec)} \odif{\DataValVec} & \le 1 \quad \textbf{(continuous case)}, \\
\end{split}
\end{equation}
whereas $\CodeLen_\ML \ab[\ProbabilityMeasure_{\ParamSpace}]$ does not satisfy it.
The parametric complexity $\Complexity \ab(\ProbabilityMeasure_{\ParamSpace})$ is defined as the minimum constant for the function $\CodeLen_{\NML}$ to be a code length function, i.e., to satisfy the Kraft-McMillan inequality \eqref{eqn:Kraft}.
Specifically, it is defined by
\begin{equation}
\Complexity \ab(\ProbabilityMeasure_{\ParamSpace})
\DefEq
\min \ab\{C \in \Real \middle| \textrm{$\CodeLen_{\ML} \ab[\ProbabilityMeasure_{\ParamSpace}] + C$ satisfies \eqref{eqn:Kraft}}\},
\end{equation}
where $\CodeLen_{\ML} \ab[\ProbabilityMeasure_{\ParamSpace}] + C$ is defined by $\ab(\CodeLen_{\ML} \ab[\ProbabilityMeasure_{\ParamSpace}] + C) \ab(\DataValVec) \DefEq \CodeLen_{\ML} \ab[\ProbabilityMeasure_{\ParamSpace}] \ab(\DataValVec) + C$.
It is known that $\Complexity \ab(\ProbabilityMeasure_{\ParamSpace})$ is given by
\begin{equation}
\Complexity \ab(\ProbabilityMeasure_{\ParamSpace})
=
\begin{cases}
\displaystyle \log_{b} \ab(\sum_{\DataValVec \in \DataSpace} \PMF \ab[\ProbabilityMeasure_{\ParamMLEVec \ab(\DataValVec)}] \ab(\DataValVec)) & \quad \textbf{(discrete case)},
\\    
\displaystyle \log_{b} \ab(\int_{\DataSpace} \PDF \ab[\ProbabilityMeasure_{\ParamMLEVec \ab(\DataValVec)}] \ab(\DataValVec) \odif{\DataValVec}) & \quad \textbf{(continuous case)}.
\\
\end{cases}
\end{equation}
In learning theory perspective, the MC is a measure of the overfitting risk of the probabilistic model $\ProbabilityMeasure_{\ParamSpace}$.
In general, the more probability measures it contains, the higher overfitting risk it has.
Since the MC $\Complexity$ is an increasing function as a set function, i.e., $\Complexity \ab(\ProbabilityMeasure_{\ParamSpace}) \le \Complexity \ab(\ProbabilityMeasure'_{\ParamSpace'})$ if $\ProbabilityMeasure_{\ParamSpace} \subset \ProbabilityMeasure'_{\ParamSpace'}$, we can say that it measures the risk of overfitting in some sense.
On the other hand, the negative-log-maximum-likelihood $\CodeLen_\ML \ab[\ProbabilityMeasure_{\ParamSpace}] \ab(\DataValVec)$ indicates how well the probabilistic model fit the data and decreasing function as a set function.
Hence, the NML code length, which is the sum of the negative-log-maximum-likelihood and the model complexity can be regarded as an index to measure how well the model works on the observed data in the trade-off between underfitting and overfitting.

In evaluating the NML code length, the evaluation of the MC $\Complexity \ab(\ProbabilityMeasure_{\ParamSpace})$ is more difficult than the log-max-likelihood $\CodeLen_\ML \ab[\ProbabilityMeasure_{\ParamSpace}] \ab(\DataValVec)$.
This is because while the closed form of the maximum likelihood for each data point suffices for evaluating the log-max-likelihood, evaluating the second term requires calculating its sum or integral over the data space $\DataSpace$.
In some discrete cases, we can directly evaluate the sum. For example, efficient dynamical programming-based algorithms have been derived for the multinomial model \cite{Kontkanen&Myllymaki:2007:Linear} and the acyclic directed graph-based na\"{i}ve Bayes model with latent variables has been observed \cite{Mononen&Myllymaki:2007:Fast}.
However, for most continuous cases, the integral in the definition of the MC is often intractable since each data point is usually a sequence and the dimensionality of the data space $\DataSpace$ is proportional to the length of each sequence. 
For example, if we consider a data sequence of two-dimensional vectors with a length of 100, we need to evaluate an integral or sum on
$\DataSpace$ as a 200-dimensional space.
For this reason, converting the sum or integral into that on the parameter space $\ParamSpace$ is a promising idea.

One significant establishment in this direction was the asymptotic approximation formula given by Rissanen \cite{Rissanen:1996}.
It can find the value of MC by evaluating the integral of the volume form determined by the Fisher information matrix for many probabilistic models under loose assumptions.
While the asymptotic formula does not give the exact value of the MC, Barron et al.\cite{barron1998minimum} ~provided a scheme to calculate the exact value of the MC for the parametric probabilistic model with sufficient statistics by the sum of a certain function, which was called the g-function later \cite{Hirai+:2013} \cite{yamanishi2023learning}, defined on the parameter space.
Although the original work by Barron et al.\cite{barron1998minimum}~was limited to the case where there exists sufficient statistics, the scheme was extended to general cases by Rissanen \cite{Rissanen:2007:Information} \cite{Rissanen:2012:Optimal}.
The idea of the scheme is to decompose the data space according to the value of the MLE.
Specifically, they consider the following decomposition for discrete parametric probabilistic model cases.
\begin{equation}
\label{eqn:DiscreteDecomposition}
\sum_{\DataValVec \in \DataSpace} \PMF \ab[\ProbabilityMeasure_{\ParamMLEVec \ab (\DataValVec)}] \ab (\DataValVec)
=
\sum_{\ParamVec \in \ParamMLEVec \ab (\DataSpace)} \sum_{\DataValVec \in \ParamMLEVec^{-1} \ab (\ab \{\ParamVec\})} \PMF \ab[\ProbabilityMeasure_{\ParamVec}] \ab (\DataValVec).
\end{equation}
Here, $\ParamMLEVec \ab (\DataSpace) \DefEq \ab \{\ParamMLEVec \ab (\DataValVec) \middle| \DataValVec \in \ParamSpace\} \subset \DataSpace$ is the image of $\DataSpace$ mapped by $\ParamMLEVec$ and $\ParamMLEVec^{-1} \ab (\ab \{\ParamVec\}) \DefEq \ab \{\DataValVec \middle| \ParamMLEVec \ab (\DataValVec) \in \ab \{\ParamVec\}\}$ is the inverse image of the one point set $\ab \{\ParamVec\}$ pulled back by $\ParamMLEVec$.
Simple calculation allows us to see that $\sum_{\DataValVec \in \ParamMLEVec^{-1} \ab (\ab \{\ParamVec\})} \PMF \ab[\ProbabilityMeasure_{\ParamVec}] \ab (\DataValVec) = \PMF \ab[\ParamMLEVec_{\sharp} \ProbabilityMeasure_{\ParamVec}] (\ParamVec)$, where $\PMF \ab[\ParamMLEVec_{\sharp} \ProbabilityMeasure_{\ParamVec}]$ is the PMF of the probability measure $\ParamMLEVec_{\sharp} \ProbabilityMeasure_{\ParamVec}$, which is the probability measure pushforwarded by the map $\ParamMLEVec$ from $\ProbabilityMeasure_{\ParamVec}$.
In other words, $\ParamMLEVec_{\sharp} \ProbabilityMeasure_{\ParamVec}$ represents the distribution of $\ParamMLEVec (\DataRVVec_{\ParamVec})$, where the distribution of the random variable $\DataRVVec_{\ParamVec}$ on $\DataSpace$ is the one determined by $\ProbabilityMeasure_{\ParamVec}$.
To wrap up, we can calculate the MC by evaluating the sum of a function defined using the PMF $\PMF \ab[\ParamMLEVec_{\sharp} \ProbabilityMeasure_{\ParamVec}]$ of the estimator $\ParamMLEVec$ as follows:
\begin{equation}
\label{eqn:DiscreteMC}
\sum_{\DataValVec \in \DataSpace} \PMF \ab[\ProbabilityMeasure_{\ParamMLEVec \ab (\DataValVec)}] \ab (\DataValVec)
=
\sum_{\ParamVec \in \ParamMLEVec \ab (\DataSpace)} \PMF \ab[\ParamMLEVec_{\sharp} \ProbabilityMeasure_{\ParamVec}] (\ParamVec).
\end{equation}
The right-hand side is often easier to evaluate than the left-hand side if the MLE's property is well-known, as in exponential family cases.
% if we can calculate the bivariate function $g: \ParamSpace \times \ParamSpace \to [0, 1]$ defined by $g \ab (\ParamVec', \ParamVec) = \PMF_{\ParamMLEVec \ab (\DataRVVec_{\ParamVec})} (\ParamVec')$, we can find the value of the MC by taking the sum of $g \ab (\ParamVec', \ParamVec')$ over $\ParamVec' \in \ParamSpace$.
% We often call $g$ the g-function.

While it provides a useful formula, the above discussion is limited to discrete cases. 
One might naturally expect similar discussions to hold true for continuous cases.
Indeed, there have been pieces of previous work (e.g., \cite{Rissanen:2000}, \cite{Hirai&Yamanishi:2013:Efficient} \cite{yamanishi2023detecting}) calculating the MC assuming that the following similar formula holds:
\begin{equation}
\label{eqn:ContinuousGFormula}
\int_{\DataSpace} \PDF \ab[\ProbabilityMeasure_{\ParamMLEVec \ab (\DataValVec)}] \ab (\DataValVec) \odif{\DataValVec}
=
\int_{\ParamMLEVec \ab (\DataSpace)} \PDF \ab[\ParamMLEVec_{\sharp} \ProbabilityMeasure_{\ParamVec}] (\ParamVec) \odif{\ParamVec},
\end{equation}
where $\PDF \ab[\ParamMLEVec_{\sharp} \ProbabilityMeasure_{\ParamVec}]$ is the PDF of the probability measure $\ParamMLEVec_{\sharp} \ProbabilityMeasure_{\ParamVec}$.
This formula is useful since the right-hand side of \eqref{eqn:ContinuousGFormula} is often easy to evaluate (see Example \ref{exm:ExponentialLMC} for a specific calculation example).
However, to the best of our knowledge, \eqref{eqn:ContinuousGFormula} has not been proved. 
The previous studies \cite{barron1998minimum} \cite{Rissanen:2007:Information} \cite{Rissanen:2012:Optimal} deriving the formula \eqref{eqn:DiscreteMC} did not claim that it can apply to the continuous cases, either.
One might expect to prove it the same way by converting the decomposition \eqref{eqn:DiscreteDecomposition} to a continuous version by replacing the sum there with the integral.
Yet, the resulting equation
\begin{equation}
\label{eqn:NaiveDecomposition}
\int_{\DataSpace} \PDF \ab[\ProbabilityMeasure_{\ParamMLEVec \ab (\DataValVec)}] \ab (\DataValVec) \odif{\DataValVec}
=
\int_{\ParamMLEVec \ab (\DataSpace)} \ab(\int_{\ParamMLEVec^{-1} \ab( \ab\{\ParamVec\})} \PDF \ab[\ProbabilityMeasure_{\ParamVec}] \ab(\DataValVec) \odif{\DataValVec}) \odif{\ParamVec} \TagIncorrect
\end{equation}
cannot straightforwardly be interpreted as a correct equation. 
Here, the symbol \Incorrect \ in an equation number in this paper indicates that the equation does not hold in general.
First, for most cases, the set $\ParamMLEVec^{-1} \ab (\ab \{\ParamVec\})$ has measure zero and hence the value of the internal integral $\int_{\ParamMLEVec^{-1} \ab (\ab \{\ParamVec\})} \PDF \ab[\ProbabilityMeasure_{\ParamVec}] \ab (\DataValVec) \odif{\DataValVec}$ is zero.
Take the normal distribution case as an example.
If the data space $\DataSpace$ is the set of sequences of  real numbers
Therefore, even for nontrivial cases where the left-hand side of \eqref{eqn:NaiveDecomposition} is nonzero, the right-hand side is zero. 
This means that \eqref{eqn:NaiveDecomposition} cannot hold in general. 
Another issue is that the left-hand side and the internal integral of the right-hand side \eqref{eqn:NaiveDecomposition} are invariant when we apply the change of the parameter variable $\ParamVec$, but the outer integral is not.
As a result, the left-hand side is independent of but the right-hand side is dependent on the parametrization.
This fact also demonstrates that \eqref{eqn:NaiveDecomposition} cannot hold in general. 
Detailed discussion will be given in Section \ref{sec:Issue}.

Although the formula \eqref{eqn:ContinuousGFormula} has not been proved, it has been used even for continuous parametric probability models (e.g., \cite{Rissanen:2000}\cite{Hirai+:2013}).
Hence, whether \eqref{eqn:ContinuousGFormula} holds or not is a practical question, though it is non-trivial.
Interestingly, some results calculated by \eqref{eqn:ContinuousGFormula} (e.g., that for exponential distributions \cite{Rissanen:2007:Information} \cite{Rissanen:2012:Optimal}) are consistent with the results of a different calculation method using the Fourier transform \cite{Suzuki&Yamanishi:2018:Exact} \cite{suzuki2021fourier}. 
Thus, the above calculation method \eqref{eqn:ContinuousGFormula}, while problematic to prove, is expected to be correct for continuous models under appropriate conditions.
Since the MC calculation method using the Fourier transform \cite{Suzuki&Yamanishi:2018:Exact} \cite{suzuki2021fourier} is more complicated than \eqref{eqn:ContinuousGFormula} in that it involves additional frequency variables, mathematical justification of the simple calculation method \eqref{eqn:ContinuousGFormula} is demanded in this field.
Hence, our research question is the following: ``\Emph{Is the MC calculation method \eqref{eqn:ContinuousGFormula} correct} (while we cannot trivially prove it)?''

This paper answers the above question \Emph{in the affirmative}.
Specifically, we, for the first time, successfully justify the MC calculation formula \eqref{eqn:ContinuousGFormula} that has been applied without theoretical rationales in previous work.

The key idea of our paper is to modify the decomposition \eqref{eqn:NaiveDecomposition}, which does not hold in general, by using the coarea formula, which has been studied in the geometric measurement theory.
According to the coarea formula, we can see that the inner integral of \eqref{eqn:NaiveDecomposition} should have been replaced by that on a Hausdorff measure instead of a Lebesgue measure, and the non-square version of Jacobian determinants should have been inserted in the outer integral to reflect the ratio between the infinitesimal change in parameter space and the infinitesimal in data space.

The contributions of the paper are listed as follows:

\begin{itemize}
    \item This paper reveals that there are issues with past proofs for the formula \eqref{eqn:ContinuousGFormula} to calculate the MC of continuous PPM and that their correction is not obvious.
    \item We provide a specific representation of the PDF of general estimators, including maximum likelihood estimators, where the data follows a continuous probability distribution.
    \item Using the PDF of the estimator given above, we derive a MC calculation formula and show that it agrees with the formula \eqref{eqn:ContinuousGFormula} used in previous work without a valid proof. This justifies for the first time the MC calculations in previous work based on \eqref{eqn:ContinuousGFormula}.
\end{itemize}

The remainder of the paper is organized as follows. 
Section \ref{sec:Definition} provides definitions regarding NML and MC.
Section \ref{sec:Issue} explains in detail the issues that arise when we attempt to apply existing proofs of the MC calculation formula to continuous cases.
Section \ref{sec:Idea} explains our core idea of new proofs of the calculation formula. 
This section includes mathematical preliminaries to state our main results, explaining why those preliminaries are necessary for our goal.
Section \ref{sec:Main} provides our main results, the MC calculation formula for continuous cases with complete proofs.

\section{Notation and Definitions}
\label{sec:Definition}
\subsection{Notation}
In this paper, $\Real$, $\Real_{\ge 0}$, and $\Real_{>0}$ denote the sets of real numbers, non-negative real numbers, and positive real numbers, respectively.
Likewise, $\Integer$, $\Integer_{\ge 0}$, and $\Integer_{>0}$ denote the sets of integers, non-negative integers, and positive integers.
We also define $\overline{\Real} \DefEq \Real \cup \{-\infty,+\infty\}$, $\overline{\Real}_{\ge 0} \DefEq \Real_{\ge 0} \cup \{+\infty\}$, and $\overline{\Real}_{> 0} \DefEq \Real_{> 0} \cup \{+\infty\}$.
For $\DataDim \in \Integer_{>0}$, we denote the $\DataDim$-dimensional real vector space by $\Real^{\DataDim}$ and the Lebesgue measure on $\Real^{\DataDim}$ by $\Lebesgue^{\DataDim}$.
\begin{remark}
In the introduction section, we used the notation $\odif{\DataValVec}$ to denote the Lebesgue integral. However, in the following, we use $\odif{\Lebesgue^{\DataDim} \ab(\DataValVec)}$ to clarify it is the Lebesgue integral since we also use another definition of integral later.
\end{remark}
The symbol \Incorrect \ in an equation number in this paper indicates that the equation does not hold in general.

\subsection{Probabilistic model}
In this paper, we denote the data space by $\DataSpace$ and assume that $\DataSpace \subset \Real^{\DataDim}$ for some $\DataDim \in \Integer_{> 0}$.
In applications, $\DataSpace$ might be the direct product of $\NData$ copies of some data space, where $\NData$ is the length of a data sequence.
For those cases, $\DataSpace$ can be regarded as the data-sequence space.

A set of distributions on an identical data space is called a \NewTerm{probabilistic model}.
Mathematically, a distribution is represented as a probability measure; hence, we use the terms ``distribution'' and ``probability measure'' in the same meaning in this paper.
Thus, we can say that a probabilistic model is a set of probability measures on some data space $\DataSpace \subset \Real^{\DataDim}$.
We parameterize the probability measures in a probabilistic model by a parameter vector $\ParamVec$ in some parameter space $\ParamSpace$.
We assume that $\ParamSpace \subset \Real^{\ParamDim}$ for $\ParamDim \in \Integer_{> 0}$.
The probability measure indicated by $\ParamVec \in \ParamSpace$ is denoted by $\ProbabilityMeasure_{\ParamVec}$.
The probabilistic model is the set $\ProbabilityMeasure_{\ParamSpace} \DefEq \ab\{\ProbabilityMeasure_{\ParamVec}\}_{\ParamVec \in \ParamSpace}$.
When a probabilistic model is parameterized like the above, the probabilistic model is called a \NewTerm{parametric probabilistic model (PPM)}.

In practice, it is sufficient to consider discrete PPMs and continuous PPMs, where the latter is the main focus of this paper.
\begin{definition}[discrete PPM and probability mass function]
A PPM is \NewTerm{discrete} if the data space $\DataSpace$ is at most countably finite.
If a PPM is discrete, for $\ParamVec \in \ParamSpace$, the function $\PMF \ab[\ProbabilityMeasure_{\ParamVec}]: \DataSpace \to \overline{\Real}_{\ge 0}$ defined by $\PMF \ab[\ProbabilityMeasure_{\ParamVec}] \ab(\DataValVec) \DefEq \ProbabilityMeasure_{\ParamVec} \ab(\ab\{\DataValVec\ab\})$ is called the \NewTerm{probability mass function (PMF)} of the probability measure $\ProbabilityMeasure_{\ParamVec}$.
\end{definition}
\begin{definition}[continuous PPM and probability density function]
A PPM is \NewTerm{continuous} if for any $\ParamVec \in \ParamSpace$, there exists a \NewTerm{probability density function (PDF)} $\PDF \ab[\ProbabilityMeasure_{\ParamVec}]: \DataSpace \to \overline{\Real}_{\ge 0}$ such that
\begin{equation}
\ProbabilityMeasure_{\ParamVec} \ab(\Set)
=
\int_{\Set} \PDF \ab[\ProbabilityMeasure_{\ParamVec}] \ab(\DataValVec) \odif{\Lebesgue^{\DataDim} \ab(\DataValVec)}
\end{equation}
for any Lebesgue measurable set $\Set \in \DataSpace$.
\end{definition}
\begin{remark}
Strictly speaking, a probability measure has multiple PDFs in general.
For example, the following two functions $\PDF_{Z}$ and $\PDF'_{Z}$ are both PDFs of the standard normal distribution:
\begin{equation}
\begin{split}
\PDF_{Z} (z) 
&\DefEq 
\frac{1}{\sqrt{2 \pi}} \exp \ab(-\frac{z^2}{2}),
\\
\PDF'_{Z} (z) 
&\DefEq 
\begin{cases}
\frac{1}{\sqrt{2 \pi}} \exp \ab(-\frac{z^2}{2}) \quad & \textrm{if $z \ne 0$},
\\
0 \quad & \textrm{if $z = 0$},
\end{cases}
\end{split}
\end{equation}
because they are equal except on $\{0\}$, which is a set with zero measure.
In the area of normalized maximum likelihood, this is an annoying issue since the likelihood may vary, depending on which PDF to choose. 
For example, the likelihood of $Z=0$ is $\frac{1}{\sqrt{2 \pi}}$ according to $\PDF_{Z}$ and $0$ according to $\PDF'_{Z}$.

Nevertheless, the non-uniqueness of PDFs of a probability measure is not a practical issue in most cases since we can usually select a canonical PDF for each random variable.
For instance, in the above example of the standard normal distribution, we usually select $\PDF_{Z}$ and not $\PDF'_{Z}$.
For this reason, in this paper, we assume that there exists a unique canonical PDF $\PDF \ab[\ProbabilityMeasure_{\ParamVec}]$ for all $\ParamVec \in \ParamSpace$.
We can guarantee the existence of such a canonical PDF if each probability measure has a continuous PDF, according to the following proposition.
\begin{proposition}
Suppose that $\BaseMeasure$ is a strictly positive measure, that is, $\BaseMeasure(\Set)>0$ holds for any open set $\Set$. 
Then, any probability measure $\ProbabilityMeasure$ has a unique continuous PDF if it exists. 
\end{proposition}
\end{remark}

\subsection{Normalized Maximum Likelihood and Model Complexity}
This paper aspires to calculate the normalized maximum likelihood (NML) code length. In particular, the model complexity, which is one term of the NML code length is the focus of this paper.
We define these amounts in this subsections.
\begin{definition}[normalized maximum likelihood (NML) and model complexity (MC)]
The \NewTerm{normalized maximum likelihood (NML) code length function} $\CodeLen_{\NML} \ab[\ProbabilityMeasure_{\ParamSpace}]: \DataSpace \to \overline{\Real}$ of a discrete or continuous PPM is defined by
\begin{equation}
\label{eqn:NLML}
\CodeLen_{\NML} \ab[\ProbabilityMeasure_{\ParamSpace}] \ab(\DataValVec)
\DefEq
\CodeLen_{\ML} \ab[\ProbabilityMeasure_{\ParamSpace}] \ab(\DataValVec) + \Complexity \ab(\ProbabilityMeasure_{\ParamSpace}),
\end{equation}
where $\CodeLen_{\ML} \ab[\ProbabilityMeasure_{\ParamSpace}]: \DataSpace \to \overline{\Real}$ is the negative-log-maximum-likelihood function defined by
\begin{equation}
\CodeLen_{\ML} \ab[\ProbabilityMeasure_{\ParamSpace}] \ab(\DataValVec)
\DefEq
\begin{cases}
- \log_{b} \ab(\sup_{\ParamVec \in \ParamSpace} \PMF \ab[\ProbabilityMeasure_{\ParamVec}] \ab(\DataValVec)) \quad & \textbf{(discrete)}, \\
- \log_{b} \ab(\sup_{\ParamVec \in \ParamSpace} \PDF \ab[\ProbabilityMeasure_{\ParamVec}] \ab(\DataValVec)) \quad & \textbf{(continuous)}, \\
\end{cases}
\end{equation}
and $\Complexity \ab(\ProbabilityMeasure_{\ParamSpace})$ is the \NewTerm{model complexity (MC)}, also known as the \NewTerm{parametric complexity}, of $\ProbabilityMeasure_{\ParamSpace}$, defined by $\Complexity \ab(\ProbabilityMeasure_{\ParamSpace}) \DefEq \log_{b} \ExponentialComplexity \ab(\ProbabilityMeasure_{\ParamSpace})$, where
\begin{equation}
\begin{split}
& \ExponentialComplexity \ab(\ProbabilityMeasure_{\ParamSpace}) 
\\
& \DefEq
\begin{cases}
\displaystyle \sum_{\DataValVec \in \DataSpace} \ab(\sup_{\ParamVec \in \ParamSpace} \PMF \ab[\ProbabilityMeasure_{\ParamVec}] \ab(\DataValVec)) \quad & \textbf{(discrete)}, \\
\displaystyle \int_{\DataSpace} \ab(\sup_{\ParamVec \in \ParamSpace} \PDF \ab[\ProbabilityMeasure_{\ParamVec}] \ab(\DataValVec)) \odif{\Lebesgue^{\DataDim} \ab(\DataValVec)} \quad & \textbf{(continuous)}. \\
\end{cases}
\end{split}
\end{equation}
Here $b > 0$ is a fixed base of the logarithm.
\end{definition}
\begin{remark}
\begin{enumerate}
    \item In this paper, we fix $b = 2$ for discrete cases and $b = e$ for continuous cases.
    \item Clearly, if we find the value of $\ExponentialComplexity \ab(\ProbabilityMeasure_{\ParamSpace})$, then we can find the value of $\Complexity \ab(\ProbabilityMeasure_{\ParamSpace})$.
    For this reason, we also call $\ExponentialComplexity \ab(\ProbabilityMeasure_{\ParamSpace})$ the model complexity (MC) of $\ProbabilityMeasure_{\ParamSpace}$ in this paper.
\end{enumerate}
\end{remark}
As we have seen, the MC is defined on the maximum likelihood. 
As we aim to convert the integral on the data space in the MC's definition to an integral on the parameter space, the maximum likelihood estimator plays an essential role as it relates the maximum likelihood and the corresponding parameter. 
\begin{definition}[maximum likelihood estimator (MLE)]
A map $\ParamMLEVec: \DataSpace \to \ParamSpace$ is called a \NewTerm{maximum likelihood estimator (MLE)} of a discrete PPM if it satisfies
\begin{equation}
\ParamMLEVec \ab(\DataValVec) = \sup_{\ParamVec \in \ParamSpace} \PMF \ab[\ProbabilityMeasure_{\ParamVec}] \ab(\DataValVec),
\end{equation}
for all $\DataValVec \in \DataSpace$.
Likewise, it is called a MLE of a continuous PPM if it satisfies
\begin{equation}
\ParamMLEVec \ab(\DataValVec) = \sup_{\ParamVec \in \ParamSpace} \PDF \ab[\ProbabilityMeasure_{\ParamVec}] \ab(\DataValVec),
\end{equation}
for all $\DataValVec \in \DataSpace$.
\end{definition}
\begin{proposition}
Suppose that $\ParamMLEVec: \DataSpace \to \ParamSpace$ is a MLE.
Then, the negative-log-maximum-likelihood function 
$\CodeLen_{\ML} \ab[\ProbabilityMeasure_{\ParamSpace}]: \DataSpace \to \overline{\Real}$ is given by
\begin{equation}
\label{eqn:MLENLML}
\CodeLen_{\ML} \ab[\ProbabilityMeasure_{\ParamSpace}] \ab(\DataValVec)
=
\begin{cases}
- \log_{b} \PMF \ab[\ProbabilityMeasure_{\ParamMLEVec \ab(\DataValVec)}] \ab(\DataValVec) \quad & \textbf{(discrete)}, \\
- \log_{b} \PDF \ab[\ProbabilityMeasure_{\ParamMLEVec \ab(\DataValVec)}] \ab(\DataValVec) \quad & \textbf{(continuous)}, \\
\end{cases}
\end{equation}
and the (exponential) model complexity $\ExponentialComplexity \ab(\ProbabilityMeasure_{\ParamSpace})$ is given by
\begin{equation}
\begin{split}
& \ExponentialComplexity \ab(\ProbabilityMeasure_{\ParamSpace}) 
\\
& =
\begin{cases}
\displaystyle \sum_{\DataValVec \in \DataSpace} \PMF \ab[\ProbabilityMeasure_{\ParamMLEVec \ab(\DataValVec)}] \ab(\DataValVec) \quad & \textbf{(discrete)}, \\
\displaystyle \int_{\DataSpace} \PDF \ab[\ProbabilityMeasure_{\ParamMLEVec \ab(\DataValVec)}] \ab(\DataValVec) \odif{\Lebesgue^{\DataDim} \ab(\DataValVec)} \quad & \textbf{(continuous)}. \\
\end{cases}
\end{split}
\end{equation}
\end{proposition}

In the following sections, we will discuss how to calculate the MC $\ExponentialComplexity \ab(\ProbabilityMeasure_{\ParamSpace})$.
There, the distribution of an estimator map, including MLE, plays a crucial role.
The probability measure of such a distribution is called the pushforward measure.
\begin{definition}[pushforward measure]
For a measurable map $\ParamMapVec: \DataSpace \to \ParamSpace$ and a probability measure $\ProbabilityMeasure$ on $\DataSpace$, we define the \NewTerm{pushforward measure} $\ParamMapVec_{\sharp} \ProbabilityMeasure$, which is a probability measure on the parameter space $\ParamSpace$, by
\begin{equation}
\ab(\ParamMapVec_{\sharp} \ProbabilityMeasure)
\ab(\Set)
=
\ProbabilityMeasure \ab(\ParamMapVec^{-1} \ab(\Set)),
\end{equation}
where $\Set \in \ParamSpace$ is an arbitrary measurable set.
Note that the right-hand side is defined since $\ParamMapVec^{-1} \ab(\Set)$ is a measurable set in $\DataSpace$, which follows because $\Set$ is a measurable set in $\ParamSpace$ and $\ParamMapVec$ is a measurable map.
\end{definition}
\begin{remark}
The pushforward measure $\ParamMapVec_{\sharp} \ProbabilityMeasure$ is nothing but the probability measure of the distribution that $\ParamMapVec \ab(\DataRVVec)$ follows, where $\DataRVVec$ follows the distribution determined by $\ProbabilityMeasure$.
\end{remark}

\subsection{Luckiness}
In practice, the sum or integral that defines the MC may diverge.
For such cases, we introduce a weight on parameter space, called the \NewTerm{luckiness}.
\begin{definition}[Luckiness NML (LNML) and Luckiness MC (LMC)]
Suppose that $\ParamMLEVec: \DataSpace \to \ParamSpace$ is a MLE and let $\Luckiness: \ParamSpace \to \Real_{\ge 0}$ be a function of a parameter $\ParamVec$.
We define the \NewTerm{luckiness normalized maximum likelihood (LNML) code length function} 
$\CodeLen_{\NML, \Luckiness} \ab[\ProbabilityMeasure_{\ParamSpace}]: \DataSpace \to \overline{\Real}$ of a discrete or continuous PPM is defined by
\begin{equation}
\CodeLen_{\NML, \Luckiness} \ab[\ProbabilityMeasure_{\ParamSpace}] \ab(\DataValVec)
\DefEq
\CodeLen_{\ML} \ab[\ProbabilityMeasure_{\ParamSpace}] \ab(\DataValVec) + \Complexity_{\Luckiness} \ab(\ProbabilityMeasure_{\ParamSpace}),
\end{equation}
where $\CodeLen_{\ML} \ab[\ProbabilityMeasure_{\ParamSpace}]: \DataSpace \to \overline{\Real}$ is the negative-log-maximum-likelihood function defined by \eqref{eqn:NLML} or \eqref{eqn:MLENLML}
and $\Complexity_{\Luckiness} \ab(\ProbabilityMeasure_{\ParamSpace})$ is defined by $\Complexity_{\Luckiness} \ab(\ProbabilityMeasure_{\ParamSpace}) \DefEq \log_{b} \ExponentialComplexity_{\Luckiness} \ab(\ProbabilityMeasure_{\ParamSpace})$, where
\begin{equation}
\begin{split}
& \ExponentialComplexity_{\Luckiness} \ab(\ProbabilityMeasure_{\ParamSpace}) 
\\
& \DefEq
\begin{cases}
\displaystyle \sum_{\DataValVec \in \DataSpace} \PMF \ab[\ProbabilityMeasure_{\ParamMLEVec \ab(\DataValVec)}] \ab(\DataValVec) \Luckiness \ab(\ParamMLEVec \ab(\DataValVec)) \quad & \textbf{(discrete)}, \\
\displaystyle \int_{\DataSpace} \PDF \ab[\ProbabilityMeasure_{\ParamMLEVec \ab(\DataValVec)}] \ab(\DataValVec) \Luckiness \ab(\ParamMLEVec \ab(\DataValVec)) \odif{\Lebesgue^{\DataDim} \ab(\DataValVec)} \quad & \textbf{(continuous)}. \\
\end{cases}
\end{split}
\end{equation}
We call $\Luckiness$ the \NewTerm{luckiness function} and $\Complexity_{\Luckiness} \ab(\ProbabilityMeasure_{\ParamSpace})$ the \NewTerm{luckiness model complexity (LMC)} of the PPM $\ProbabilityMeasure_{\ParamSpace}$.
Here $b > 0$ is a fixed base of the logarithm.
\end{definition}
\begin{remark}
Define a constant function $1_{\ParamSpace}: \ParamSpace \to \Real$ by $1_{\ParamSpace} \ab(\ParamVec) = 1$. Then, the LNML $\CodeLen_{\NML, 1_{\ParamSpace}} \ab[\ProbabilityMeasure_{\ParamSpace}]$ and LMC $\Complexity_{1_{\ParamSpace}} \ab(\ProbabilityMeasure_{\ParamSpace})$ equal the NML $\CodeLen_{\NML} \ab[\ProbabilityMeasure_{\ParamSpace}]$ and MC $\Complexity \ab(\ProbabilityMeasure_{\ParamSpace})$, respectively.
In this sense, the LNML and LMC are generalizations of the NML and MC, respectively.
Hence, we discuss the LNML and LMC in the remainder of the paper.
We can always obtain the discussion on the NML and MC by substituting $\Luckiness = 1_{\ParamSpace}$. 

We often consider LNML and LMC when we want to limit the parameter space to its subspace for defining the MC as the MC would diverge to infinity otherwise.
Specifically, if we limit the parameter space for defining the MC to $\Set$, we use the indicator function $1_{\Set}$ as a luckiness function, where $1_{\Set}: \ParamSpace \to \Real_{\ge 0}$ is defined by
\begin{equation}
1_{\Set} \ab(\ParamVec)
\DefEq
\begin{cases}
1 & \quad \textrm{if $\ParamVec \in \Set$},\\
0 & \quad \textrm{if $\ParamVec \notin \Set$}.\\
\end{cases}
\end{equation}
\end{remark}

\section{Limitation of Previous Work's Derivation}
\label{sec:Issue}

Non-asymptotic formulae to calculate the MC of a discrete PPM have been derived by \cite{barron1998minimum} and \cite{Rissanen:2007:Information} \cite{Rissanen:2012:Optimal}.
In this section, we review those derivations and clarify why we cannot modify those derivations to continuous PPM cases straightforwardly.
Note that we omit the discussion on luckiness for simplicity, while a similar discussion holds even with luckiness.
For both derivations, the following decomposition is a presumption.
\begin{lemma}
\label{lem:DiscreteGeneralDecomposition}
Suppose that $\DataSpace$ is at most countably infinite set, $\FuncIII: \DataSpace \to \Real$ is a function on $\DataSpace$, and $\ParamMapVec: \DataSpace \to \ParamVec$ is an estimator map.
Then, we have the following:
\begin{equation}
\begin{split}
\sum_{\DataValVec \in \DataSpace} \FuncIII \ab (\DataValVec)
& =
\sum_{\ParamVec \in \ParamMapVec \ab (\DataSpace)} \sum_{\DataValVec \in \ParamMapVec^{-1} \ab (\ab \{\ParamVec\})} \FuncIII \ab (\DataValVec),    
\end{split}
\end{equation}
if either hand of the equation absolutely converges, i.e., if $\sum_{\DataValVec \in \DataSpace} \ab|\FuncIII \ab (\DataValVec)|$ converges or if $\sum_{\ParamVec \in \ParamMapVec \ab (\DataSpace)} \sum_{\DataValVec \in \ParamMapVec^{-1} \ab (\ab \{\ParamVec\})} \ab|\FuncIII \ab (\DataValVec)|$ converges.
\end{lemma}
\begin{remark}
In Lemma \ref{lem:DiscreteGeneralDecomposition}, since we assumed that $\DataSpace$ is at most countably infinite, $\ParamMapVec \ab(\DataSpace) \DefEq \ab\{\ParamMapVec \ab(\DataValVec) \middle| \DataValVec \in \DataSpace\}$ and $\ParamMapVec^{-1} \ab(\ab\{\ParamVec\}) \DefEq \ab\{\DataValVec \middle|\ParamMapVec \ab(\DataValVec) \in \ab\{\ParamVec\} \}$ are both at most countable infinite.
Hence, all sums in Lemma \ref{lem:DiscreteGeneralDecomposition} are defined as the sum of a series.
In particular, it does not include the sum of an uncountable set, whose definition is non-trivial.
\end{remark}
Lemma \ref{lem:DiscreteGeneralDecomposition} itself is trivial since it only performs a reordering operation on a (at most) infinite series, and any reordering operation on an absolutely convergent infinite series preserves its value.

Below, we will further explain how the above decomposition was used to derive the MC calculation formula for discrete cases in \cite{barron1998minimum} and \cite{Rissanen:2007:Information} \cite{Rissanen:2012:Optimal} and why we cannot trivially convert the discussion into continuous cases.

\subsection{Derivation by Barron et al.~\cite{barron1998minimum}}
The MC calculation formula by Barron et al.~\cite{barron1998minimum}\footnote{For convenience, we call the derivation here Barron et al.'s derivation and the other derivation introduced in the next subsection Rissanen's derivation. However, we also remark that Rissanen is also in the author list of the original paper of the derivation introduced in this subsection.} can apply to a discrete PPM with sufficient statistics.
Suppose that a discrete PPM has sufficient statistics $\SufficientStatVec: \DataSpace \to \Real^{\ParamDim_{\SufficientSymbol}}$.
Then, conditional PMF $\PMF \ab[\ProbabilityMeasure_{\ParamVec}|\SufficientStatVec_{\sharp} \ProbabilityMeasure_{\ParamVec}]$ defined by 
\begin{equation}
\begin{split}
\PMF \ab[\ProbabilityMeasure_{\ParamVec}|\SufficientStatVec_{\sharp} \ProbabilityMeasure_{\ParamVec}] \ab (\DataValVec|\SufficientValVec) 
& =
\frac{\ProbabilityMeasure_{\ParamVec} \ab (\ab \{\DataValVec\} \cap \SufficientStatVec^{-1} \ab (\ab \{\SufficientValVec\}))}{\ProbabilityMeasure_{\ParamVec} \ab (\SufficientStatVec^{-1} \ab (\ab \{\SufficientValVec\}))}
\\
& =
\frac{\Probability (\DataRVVec_{\ParamVec} = \DataValVec \land \SufficientStatVec (\DataRVVec_{\ParamVec}) = \SufficientValVec)}{\Probability (\SufficientStatVec (\DataRVVec) = \SufficientValVec)}
\end{split}
\end{equation}
does not depend on $\ParamVec$.
Hence, there exists a bivariate function $\PMF \ab[\ProbabilityMeasure|\SufficientStatVec_{\sharp} \ProbabilityMeasure]$ that satisfies 
\begin{equation}
\PMF \ab[\ProbabilityMeasure_{\ParamVec}|\SufficientStatVec_{\sharp} \ProbabilityMeasure_{\ParamVec}] \ab (\DataValVec|\SufficientValVec) 
=
\PMF \ab[\ProbabilityMeasure|\SufficientStatVec_{\sharp} \ProbabilityMeasure] \ab (\DataValVec|\SufficientValVec) 
\end{equation}
for all $\ParamVec \in \ParamSpace$, $\DataValVec \in \DataSpace$, and $\SufficientValVec \in \SufficientStatVec (\DataSpace)$.

The following two Lemmata are the starting point of the discussion in \cite{barron1998minimum}.

\begin{lemma}
\label{lem:DiscreteSufficientStatDecomposition}
Suppose that $\ab \{\ProbabilityMeasure_{\ParamVec}\}_{\ParamVec \in \ParamSpace}$ is a PPM and $\SufficientStatVec: \DataSpace \to \Real^{\ParamDim_{\SufficientSymbol}}$ is a sufficient statistic of $\ab \{\ProbabilityMeasure_{\ParamVec}\}_{\ParamVec \in \ParamSpace}$.
Then, we can decompose the PMF $\PMF \ab[\ProbabilityMeasure_{\ParamVec}]$ as follows:
\begin{equation}
\PMF \ab[\ProbabilityMeasure_{\ParamVec}] (\DataValVec)
=
\PMF \ab[\ProbabilityMeasure|\SufficientStatVec_{\sharp} \ProbabilityMeasure] \ab (\DataValVec|\SufficientStatVec (\DataValVec))
\PMF \ab[\SufficientStatVec_{\sharp} \ProbabilityMeasure_{\ParamVec}] (\SufficientStatVec (\DataValVec)).
\end{equation}
Here, $\PMF \ab[\SufficientStatVec_{\sharp} \ProbabilityMeasure_{\ParamVec}]$ is the PMF of the probability measure $\SufficientStatVec_{\sharp} \ProbabilityMeasure_{\ParamVec}$, which is push-forwarded by $\SufficientStatVec$ from $\ProbabilityMeasure_{\ParamVec}$; in other words, $\SufficientStatVec_{\sharp} \ProbabilityMeasure_{\ParamVec}$ is the probability measure of the random variable $\SufficientStatVec (\DataRVVec_{\ParamVec})$, where the probability measure determined by $\DataRVVec_{\ParamVec}$ is $\ProbabilityMeasure_{\ParamVec}$.
\end{lemma}

\begin{lemma}
\label{lem:SufficientStatMLE}
Suppose that $\ab \{\ProbabilityMeasure_{\ParamVec}\}_{\ParamVec \in \ParamSpace}$ is a discrete PPM and $\SufficientStatVec: \DataSpace \to \Real^{\ParamDim_{\SufficientSymbol}}$ is a sufficient statistic of $\ab \{\ProbabilityMeasure_{\ParamVec}\}_{\ParamVec \in \ParamSpace}$.
Then, there exists a function $\ParamMLEVec_{\SufficientStatVec}: \SufficientStatVec (\DataSpace) \to \ParamSpace$ such that $\ParamMLEVec_{\SufficientStatVec} \ab (\SufficientStatVec \ab (\DataValVec)) = \ParamMLEVec \ab (\DataValVec)$ holds for all $\DataValVec \in \DataSpace$.
\end{lemma}
Using the above Lemmata, \cite{barron1998minimum} derived a formula to calculate the MC as follows.
\begin{theorem}
\label{thm:BarronG}
Suppose that $\ab \{\ProbabilityMeasure_{\ParamVec}\}_{\ParamVec \in \ParamSpace}$ is a discrete PPM and $\SufficientStatVec: \DataSpace \to \Real^{\ParamDim_{\SufficientSymbol}}$ is a sufficient statistic of $\ab \{\ProbabilityMeasure_{\ParamVec}\}_{\ParamVec \in \ParamSpace}$.
Then, the following holds:
\begin{equation}
\sum_{\DataValVec \in \DataSpace} \PMF \ab[\ProbabilityMeasure_{\ParamMLEVec \ab (\DataValVec)}] \ab (\DataValVec)
=
\sum_{\SufficientValVec \in \SufficientStatVec \ab (\DataSpace)} \PMF \ab[\SufficientStatVec_{\sharp} \ProbabilityMeasure_{\ParamMLEVec_{\SufficientStatVec} \ab (\SufficientValVec)}] \ab (\SufficientValVec),
\end{equation}
where $\PMF \ab[\ParamMLEVec_{\sharp} \ProbabilityMeasure_{\ParamVec}]$ is the PMF of the probability measure $\ParamMLEVec_{\sharp} \ProbabilityMeasure_{\ParamVec}$, which is the probability measure pushforwarded by the map $\ParamMLEVec_{\sharp}$ from $\ProbabilityMeasure_{\ParamVec}$.
\end{theorem}
In the original paper \cite{barron1998minimum}, only the outline of the derivation is given and details are omitted. 
Some existing papers \cite{Rissanen:2000} \cite{Hirai&Yamanishi:2013:Efficient} seem to follow similar ideas as they mention Lemma \ref{lem:DiscreteSufficientStatDecomposition}, but the explanations of the derivation there are even rougher.
However, since our objective is to discuss the details of the proof, below we provide a complete proof based on the outline given there.
\begin{proof}
By Lemma \ref{lem:DiscreteSufficientStatDecomposition}, we have that
\begin{equation}
\begin{split}
\sum_{\DataValVec \in \DataSpace} \PMF \ab[\ProbabilityMeasure_{\ParamMLEVec \ab (\DataValVec)}] \ab (\DataValVec)
& =
\sum_{\DataValVec \in \DataSpace} \PMF \ab[\ProbabilityMeasure|\SufficientStatVec_{\sharp} \ProbabilityMeasure] \ab (\DataValVec|\SufficientStatVec (\DataValVec))
\PMF \ab[\SufficientStatVec_{\sharp} \ProbabilityMeasure_{\ParamMLEVec \ab (\DataValVec)}] (\SufficientStatVec (\DataValVec))
\\
& =
\sum_{\DataValVec \in \DataSpace} \PMF \ab[\ProbabilityMeasure|\SufficientStatVec_{\sharp} \ProbabilityMeasure] \ab (\DataValVec|\SufficientStatVec (\DataValVec))
\PMF \ab[\SufficientStatVec_{\sharp} \ProbabilityMeasure_{\ParamMLEVec_{\SufficientStatVec} \ab (\SufficientStatVec \ab (\DataValVec))}] (\SufficientStatVec \ab (\DataValVec)),
\end{split}
\end{equation}
where we applied Lemma \ref{lem:SufficientStatMLE} to get the second equation.
Apply Lemma \ref{lem:DiscreteGeneralDecomposition}, then we get
\begin{equation}
\begin{split}
& \sum_{\DataValVec \in \DataSpace} \PMF \ab[\ProbabilityMeasure|\SufficientStatVec_{\sharp} \ProbabilityMeasure] \ab (\DataValVec|\SufficientStatVec (\DataValVec))
\PMF \ab[\SufficientStatVec_{\sharp} \ProbabilityMeasure_{\ParamMLEVec_{\SufficientStatVec} \ab (\SufficientStatVec \ab (\DataValVec))}] (\SufficientStatVec (\DataValVec))
\\
&=
\sum_{\SufficientValVec \in \SufficientStatVec \ab (\DataSpace)} \sum_{\DataValVec \in \SufficientStatVec^{-1} \ab (\ab \{\SufficientValVec\})} \PMF \ab[\ProbabilityMeasure|\SufficientStatVec_{\sharp} \ProbabilityMeasure] \ab (\DataValVec|\SufficientStatVec \ab (\DataValVec))
\PMF \ab[\SufficientStatVec_{\sharp} \ProbabilityMeasure_{\ParamMLEVec_{\SufficientStatVec} \ab (\SufficientStatVec \ab (\DataValVec))}] \ab (\SufficientStatVec \ab (\DataValVec))
\\
&=
\sum_{\SufficientValVec \in \SufficientStatVec \ab (\DataSpace)} \sum_{\DataValVec \in \SufficientStatVec^{-1} \ab (\ab \{\SufficientValVec\})} \PMF \ab[\ProbabilityMeasure|\SufficientStatVec_{\sharp} \ProbabilityMeasure] \ab (\DataValVec|\SufficientValVec)
\PMF \ab[\SufficientStatVec_{\sharp} \ProbabilityMeasure_{\ParamMLEVec_{\SufficientStatVec} \ab (\SufficientValVec)}] \ab (\SufficientValVec),
\end{split}
\end{equation}
where the second equation holds because $\DataValVec \in \SufficientStatVec^{-1} \ab (\ab \{\SufficientValVec\}) \Leftrightarrow \SufficientStatVec \ab (\DataValVec) = \SufficientValVec$.
Once we obtain the above, since $\PMF \ab[\SufficientStatVec_{\sharp} \ProbabilityMeasure_{\ParamMLEVec_{\SufficientStatVec} \ab (\SufficientValVec)}] \ab (\SufficientValVec)$ does not depend on $\DataValVec$, we can calculate the right-hand side as follows:
\begin{equation}
\begin{split}
&\sum_{\SufficientValVec \in \SufficientStatVec \ab (\DataSpace)} \sum_{\DataValVec \in \SufficientStatVec^{-1} \ab (\ab \{\SufficientValVec\})} \PMF \ab[\ProbabilityMeasure|\SufficientStatVec_{\sharp} \ProbabilityMeasure] \ab (\DataValVec|\SufficientValVec)
\PMF \ab[\SufficientStatVec_{\sharp} \ProbabilityMeasure_{\ParamMLEVec_{\SufficientStatVec} \ab (\SufficientValVec)}] \ab (\SufficientValVec)
\\
&=
\sum_{\SufficientValVec \in \SufficientStatVec \ab (\DataSpace)} \PMF \ab[\SufficientStatVec_{\sharp} \ProbabilityMeasure_{\ParamMLEVec_{\SufficientStatVec} \ab (\SufficientValVec)}] \ab (\SufficientValVec) \sum_{\DataValVec \in \SufficientStatVec^{-1} \ab (\ab \{\SufficientValVec\})} \PMF \ab[\ProbabilityMeasure|\SufficientStatVec_{\sharp} \ProbabilityMeasure] \ab (\DataValVec|\SufficientValVec)
\\
&=
\sum_{\SufficientValVec \in \SufficientStatVec \ab (\DataSpace)} \PMF \ab[\SufficientStatVec_{\sharp} \ProbabilityMeasure_{\ParamMLEVec_{\SufficientStatVec} \ab (\SufficientValVec)}] \ab (\SufficientValVec),
\end{split}
\end{equation}
where the second equation holds because 
\begin{equation}
\label{eqn:SufficientConditionalSum}
\sum_{\DataValVec \in \SufficientStatVec^{-1} \ab (\ab \{\SufficientValVec\})} \PMF \ab[\ProbabilityMeasure|\SufficientStatVec_{\sharp} \ProbabilityMeasure] \ab (\DataValVec|\SufficientValVec) = 1
\end{equation}
holds for any $\SufficientValVec \in \SufficientStatVec \ab (\DataSpace)$ since $\PMF \ab[\ProbabilityMeasure|\SufficientStatVec_{\sharp} \ProbabilityMeasure] \ab (\DataValVec|\SufficientValVec) = 0$ for any $\DataValVec \notin \SufficientStatVec^{-1} \ab (\ab \{\SufficientValVec\})$.
The above completes the proof.
\end{proof}
Recall that the above proof is designed for discrete PPM cases.
Our na\"{i}ve expectation is to obtain a similar formula for continuous PPM cases by replacing the PMF with PDFs and the sums with integrals.
Unfortunately, the above derivation uses the properties listed below, which would not hold if we applied such replacements:
\begin{itemize}
    \item Lemma \ref{lem:DiscreteGeneralDecomposition},
    \item Lemma \ref{lem:DiscreteSufficientStatDecomposition}, and
    \item Equation \eqref{eqn:SufficientConditionalSum}.
\end{itemize}
Specifically, none of the decomposition
\begin{equation}
\label{eqn:NaiveMapDecomposition}
\begin{split}
& \int_{\DataSpace} \FuncIII \ab (\DataValVec) \odif{\Lebesgue^{\DataDim} (\DataValVec)}
\\
& =
\int_{\ParamMLEVec \ab (\DataSpace)} \ab (\int_{\ParamMLEVec^{-1} \ab (\ab \{\ParamVec\})} \FuncIII \ab (\DataValVec) \odif{\Lebesgue^{\ParamDim} (\DataValVec)}) \odif{\Lebesgue^{\DataDim} (\ParamVec)},
\end{split}
\TagIncorrect
\end{equation}
which is the continuous version of the decomposition \ref{lem:DiscreteGeneralDecomposition}, the equation
\begin{equation}
\label{eqn:NaiveSufficientStatDecomposition}
\PDF \ab[\ProbabilityMeasure_{\ParamVec}] (\DataValVec)
=
\PDF \ab[\ProbabilityMeasure|\SufficientStatVec_{\sharp} \ProbabilityMeasure] \ab (\DataValVec|\SufficientStatVec (\DataValVec))
\PDF \ab[\SufficientStatVec_{\sharp} \ProbabilityMeasure_{\ParamVec}] (\SufficientStatVec (\DataValVec)),
\TagIncorrect
\end{equation}
which is the continuous version of Lemma \ref{lem:DiscreteSufficientStatDecomposition},
and equation
\begin{equation}
\label{eqn:NaiveSufficientSum}
\int_{\SufficientStatVec^{-1} \ab (\ab \{\SufficientValVec\})} \PMF \ab[\ProbabilityMeasure|\SufficientStatVec_{\sharp} \ProbabilityMeasure] \ab (\DataValVec|\SufficientValVec) \odif{\Lebesgue^{\DataDim} \ab (\DataValVec)} = 1,
\TagIncorrect
\end{equation}
do not hold in general. 
The following is a counterexample.
\begin{example}[\eqref{eqn:NaiveSufficientStatDecomposition} \eqref{eqn:NaiveMapDecomposition} \eqref{eqn:NaiveSufficientSum} do not hold in general: I.]
\label{exm:Exponential}
Take the exponential distribution model as a PPM example, where the PDF is given by $\PDF \ab[\ProbabilityMeasure_{\theta}] \ab(\DataVal) = \frac{1}{\theta} \exp \ab (- \frac{\DataVal}{\theta})$ on data space $\DataSpace = \Real_{\ge 0}$. 
The parameter $\theta$ is in the parameter space $\Real_{>0}$ and the MLE $\hat{\mathrm{\theta}}$ given by $\hat{\mathrm{\theta}} (\DataVal) = \DataVal$ is a sufficient statistic.
Hence, $\hat{\mathrm{\theta}}$ should play a role of $\SufficientStatVec$ in \eqref{eqn:NaiveSufficientStatDecomposition} this section.
In the following, we see that none of \eqref{eqn:NaiveSufficientStatDecomposition} \eqref{eqn:NaiveMapDecomposition} and \eqref{eqn:NaiveSufficientSum} hold in this example.

Regarding the decomposition \eqref{eqn:NaiveMapDecomposition}, we remark that, for any $\theta \in \Real_{>0}$, we have that $\hat{\theta}^{-1} \ab (\{\theta\}) = \{\theta\}$, which is an one-point set.
Hence, the value of the inner integral in \eqref{eqn:NaiveMapDecomposition} is zero.
This means that the right-hand side of \eqref{eqn:NaiveMapDecomposition} is zero.
Since the left-hand side can be non-zero, \eqref{eqn:NaiveMapDecomposition} is wrong.

Regarding the equation \eqref{eqn:NaiveSufficientStatDecomposition}, in the first place, we cannot define the conditional PDF $\PDF \ab[\ProbabilityMeasure|\hat{\mathrm{\theta}}_{\sharp} \ProbabilityMeasure]$ since the map $\hat{\mathrm{\theta}}$ is deterministic.
Specifically, the density function $\PDF \ab[\ProbabilityMeasure|\hat{\mathrm{\theta}}_{\sharp} \ProbabilityMeasure] \ab (\cdot \middle|\theta)$ should behave like Dirac's delta ``function'' that takes a positive infinity value at $\theta$ and zero everywhere else since $\hat{\mathrm{\theta}} (\DataVal) = \DataVal$ always holds.
Hence, we cannot define $\PDF \ab[\ProbabilityMeasure|\hat{\mathrm{\theta}}_{\sharp} \ProbabilityMeasure] \ab (\cdot \middle|\theta)$ as a ``proper'' function.

Regarding \eqref{eqn:NaiveSufficientSum}, since the integral is defined on the one-point set, the left-hand side is zero. Therefore the equation \eqref{eqn:NaiveSufficientSum} does not hold.
\end{example}

\begin{example}[\eqref{eqn:NaiveSufficientStatDecomposition} \eqref{eqn:NaiveMapDecomposition} \eqref{eqn:NaiveSufficientSum} do not hold in general: II.]
\label{exm:Ellipse}
Take a zero-centered normal distribution model as a PPM example, where the PDF is given by $\PDF \ab[\ProbabilityMeasure_{\theta}] \ab(\DataValVec) = \frac{1}{\pi \Param} \exp \ab (- \frac{\ab(\DataVal_{1})^{2} + 4 \ab(\DataVal_{2})^{2}}{\Param})$ on data space $\DataSpace = \Real^{2}$. 
This example corresponds to Fig.~\ref{fig:coarea}.
The parameter $\Param$ is in the parameter space $\Real_{>0}$ and the MLE $\ParamMLE$ given by $\ParamMLE (\DataValVec) = \ab(\DataVal_{1})^{2} + 4 \ab(\DataVal_{2})^{2}$ is a sufficient statistic.
Hence, $\ParamMLE$ should play a role of $\SufficientStatVec$ in \eqref{eqn:NaiveSufficientStatDecomposition} this section.
In the following, we see that none of \eqref{eqn:NaiveSufficientStatDecomposition} \eqref{eqn:NaiveMapDecomposition} and \eqref{eqn:NaiveSufficientSum} hold in this example.

Regarding the decomposition \eqref{eqn:NaiveMapDecomposition}, for any $\Param \in \Real_{>0}$, we have that $\ParamMLE^{-1} \ab (\{\Param\}) = \ab\{\DataValVec = \begin{bmatrix}\DataVal_{1} & \DataVal_{2} \end{bmatrix}^\top \middle| \ab(\DataVal_{1})^{2} + 4 \ab(\DataVal_{2})^{2} = \Param\}$, which is an ellipse.
For example, the ellipse $\ParamMLE^{-1} \ab (\{0.2\})$ is the one drawn by the thick curve in Fig.~\ref{fig:coarea}.
An ellipse is measure zero when it is measured by the two-dimensional Lebesgue measure.
Hence, the value of the inner integral in \eqref{eqn:NaiveMapDecomposition} is zero.
This means that the right-hand side of \eqref{eqn:NaiveMapDecomposition} is zero.
Since the left-hand side can be non-zero, \eqref{eqn:NaiveMapDecomposition} is wrong.

Regarding the equation \eqref{eqn:NaiveSufficientStatDecomposition}, again, we cannot define the conditional PDF $\PDF \ab[\ProbabilityMeasure|\ParamMLE_{\sharp} \ProbabilityMeasure]$ since the map $\ParamMLE_{\sharp}$ is deterministic.
Specifically, the conditional ``density'' $\PDF \ab[\ProbabilityMeasure|\ParamMLE_{\sharp} \ProbabilityMeasure] \ab (\cdot \middle|\Param)$ should concentrate on the ellipse $\ParamMLE^{-1} \ab (\{\Param\})$ and the density should be zero everywhere else since $\ParamMLE (\DataVal) = \Param$ always holds.
Hence, we cannot define $\PDF \ab[\ProbabilityMeasure|\ParamMLE_{\sharp} \ProbabilityMeasure] \ab (\cdot \middle|\Param)$ as a ``proper'' function.

Regarding \eqref{eqn:NaiveSufficientSum}, since the integral is defined on the ellipse, the left-hand side is zero. Therefore the equation \eqref{eqn:NaiveSufficientSum} does not hold.
\end{example}

As the above examples show, equations \eqref{eqn:NaiveSufficientStatDecomposition} \eqref{eqn:NaiveMapDecomposition} and \eqref{eqn:NaiveSufficientSum} do not hold in general.
Hence, we cannot straightforwardly convert the discussion used to prove Theorem \ref{thm:BarronG} from discrete cases to continuous cases.

\subsection{Derivation by Rissanen \cite{Rissanen:2007:Information} \cite{Rissanen:2012:Optimal}}
While Barron et al.'s discussion in the previous subsection only applies when the PPM has sufficient statistics, Rissanen \cite{Rissanen:2007:Information} \cite{Rissanen:2012:Optimal} derived MC calculation formula for more general cases.
Details are also given in  \cite{yamanishi2023learning}. 
The following is the formal statement of the formula.
\begin{theorem}
\label{thm:YamanishiG}
Suppose that $\ab \{\ProbabilityMeasure_{\ParamVec}\}_{\ParamVec \in \ParamSpace}$ is a discrete PPM and let $\ParamMapVec: \DataSpace \to \ParamSpace \subset \Real^{\ParamDim}$ be an arbitrary map.
Then, the following holds:
\begin{equation}
\sum_{\DataValVec \in \DataSpace} \PMF \ab[\ProbabilityMeasure_{\ParamMapVec \ab (\DataValVec)}] \ab (\DataValVec)
=
\sum_{\ParamVec \in \ParamMapVec \ab (\DataSpace)} \PMF \ab[\ParamMapVec_{\sharp} \ProbabilityMeasure_{\ParamVec}] \ab (\ParamVec),
\end{equation}
where $\PMF \ab[\ParamMapVec_{\sharp} \ProbabilityMeasure_{\ParamVec}]$ is the PMF of the probability measure $\ParamMapVec_{\sharp} \ProbabilityMeasure_{\ParamVec}$, which is the probability measure pushforwarded by the map $\ParamMapVec_{\sharp}$ from $\ProbabilityMeasure_{\ParamVec}$.
\end{theorem}
\begin{remark}
In Theorem \ref{thm:YamanishiG}, the estimator map $\ParamMapVec$ does not have to be the maximum likelihood estimator or sufficient statistics, while this fact was not emphasized in the original context \cite{Rissanen:2007:Information} \cite{Rissanen:2012:Optimal}.
Since it does not require $\ParamMapVec$ to be sufficient statistics, Theorem \ref{thm:YamanishiG} can be more widely applicable than Theorem \ref{thm:BarronG}.
\end{remark}
In fact the proof of Theorem \ref{thm:YamanishiG} is simpler than \ref{thm:BarronG}, almost immediate from Lemma \ref{lem:DiscreteGeneralDecomposition}.
\begin{proof}
Apply Lemma \ref{lem:DiscreteGeneralDecomposition}, then we have    
\begin{equation}
\begin{split}
\sum_{\DataValVec \in \DataSpace} \PMF \ab[\ProbabilityMeasure_{\ParamMapVec \ab (\DataValVec)}] \ab (\DataValVec)
&=
\sum_{\ParamVec \in \ParamMapVec \ab (\DataSpace)}
\sum_{\DataValVec \in \ParamMapVec^{-1} \ab (\ab \{\ParamVec\})} \PMF \ab[\ProbabilityMeasure_{\ParamMapVec \ab (\DataValVec)}] \ab (\DataValVec)
\\
&=
\sum_{\ParamVec \in \ParamMapVec \ab (\DataSpace)}
\sum_{\DataValVec \in \ParamMapVec^{-1} \ab (\ab \{\ParamVec\})} \PMF \ab[\ProbabilityMeasure_{\ParamVec}] \ab (\DataValVec)
\\
&=
\sum_{\ParamVec \in \ParamMapVec \ab (\DataSpace)} \PMF \ab[\ParamMapVec_{\sharp} \ProbabilityMeasure_{\ParamVec}] \ab (\ParamVec),
\end{split}
\end{equation}
where the second equation holds since $\DataValVec \in \ParamMapVec^{-1} \ab (\ab \{\ParamVec\}) \Leftrightarrow \ProbabilityMeasure_{\ParamMapVec \ab (\DataValVec)} = \ParamVec$.
The third equation holds because
\begin{equation}
\begin{split}
\PMF \ab[\ParamMapVec_{\sharp} \ProbabilityMeasure_{\ParamVec}] \ab (\ParamVec)
&=
\ParamMapVec_{\sharp} \ProbabilityMeasure_{\ParamVec} \ab (\ab \{\ParamVec\})
\\
&=
\ProbabilityMeasure_{\ParamVec} \ab (\ParamMapVec^{-1} \ab (\ab \{\ParamVec\}))
\\
&=
\sum_{\DataValVec \in \ParamMapVec^{-1} \ab (\ab \{\ParamVec\})} \PMF \ab[\ProbabilityMeasure_{\ParamVec}] \ab (\DataValVec).
\end{split}
\end{equation}
\end{proof}

Again, the above proof is designed for discrete PPM cases, since it essentially depends on Lemma \ref{lem:DiscreteGeneralDecomposition}. 
Owing to the dependency on Lemma \ref{lem:DiscreteGeneralDecomposition}, we cannot straightforwardly modify the proof for continuous cases because its continuous version \eqref{eqn:NaiveMapDecomposition} does not hold in general, as discussed in the previous subsection.

\begin{figure*}
    \centering
\begin{tikzpicture}[scale=10.]

    % 軸を描く
    \draw[->,thick,gray!30] (-0.8, 0) -- (0.8, 0) node[black,right] {$x_{1}$};
    \draw[->,thick,gray!30] (0, -0.4) -- (0, 0.4) node[black,above] {$x_{2}$};

    % vのリストを定義
    \def\vs{0.1,0.2,0.3,0.4,0.5}

    % 色塗り範囲を描く
    \fill[gray!30] plot[domain=0:360, samples=100] ({sqrt(0.2)*cos(\x)}, {sqrt(0.2)/2*sin(\x)}) 
                   -- plot[domain=360:0, samples=100] ({sqrt(0.3)*cos(\x)}, {sqrt(0.3)/2*sin(\x)}) 
                   -- cycle;

    % v = 0.1, 0.2, ..., 1.0 の楕円を点線で描く
    \foreach \v in {0.1, 0.4, 0.5} {
        \draw[dashed, gray!80] plot[domain=0:360, samples=100] ({sqrt(\v)*cos(\x)}, {sqrt(\v)/2*sin(\x)});
    }

    % annotating each ellipse
    \foreach \v in {0.1, 0.2, 0.3, 0.4, 0.5} {
        \path (0, -{sqrt(\v)/2}) node[gray] {\small $\ParamMLE^{-1} \ab(\ab\{\v\})$};
    }

    % v = 0.2 と v = 0.3 の楕円を実線で描く
    \draw[ultra thick, black!60] plot[domain=0:360, samples=100] ({sqrt(0.2)*cos(\x)}, {sqrt(0.2)/2*sin(\x)});
    \draw[dashed, thick, black!60] plot[domain=0:360, samples=100] ({sqrt(0.3)*cos(\x)}, {sqrt(0.3)/2*sin(\x)});

    % x軸上の間隔を描く
    \draw[<->, very thick] ({sqrt(0.2)}, 0) node[below right] {\tiny $\displaystyle \ab[J \hat{\mathrm{v}} \ab(\DataValVec^{[1]})]^{-1}$} -- ({sqrt(0.3)}, 0);

    \fill ({sqrt(0.2)}, 0) circle (0.01) node[left] {$\DataValVec^{[1]}$};

    % y軸上の間隔を描く
    \draw[<->, very thick] (0, {sqrt(0.2)/2}) -- (0, {sqrt(0.3)/2}) node[midway, right] {\tiny $\displaystyle \ab[J \hat{\mathrm{v}} \ab(\DataValVec^{[2]})]^{-1}$};

    \fill (0, {sqrt(0.2)/2}) circle (0.01) node[below] {$\DataValVec^{[2]}$};

\end{tikzpicture}
\caption{Example of the coarea formula by an estimator map. In this example, $\DataSpace = \Real^{2}$, $\ParamMLEVec \ab (\DataValVec) = \DataVal_{1}^{2} + 4 \DataVal_{2}^{2}$. For $\int_{\Set} \FuncIII (\DataValVec) \odif{\Lebesgue^{\DataDim} \ab (\DataValVec)} = \int_{\Real^{\ParamDim}} \FuncII (\ParamVec) \odif{\Lebesgue^{\ParamDim} \ab (\ParamVec)}$ to hold, the integral of $\FuncIII$ on the gray area, which is bounded by $S_{\Param} = \ParamMLE^{-1} \ab (\ab \{\Param\})$ and $S_{\Param + \Delta \Param} = \ParamMLE^{-1} \ab (\ab \{\Param + \Delta \Param\})$ should be approximated by $\FuncII (\ParamVec) \Delta \Param$.
Here, the interval between $S_{\Param} = \ParamMLE^{-1} \ab (\ab \{\Param\})$ and $S_{\Param + \Delta \Param} = \ParamMLE^{-1} \ab (\ab \{\Param + \Delta \Param\})$ is not uniform everywhere but approximated by $\ab (J \ParamMLE \ab (\DataValVec))^{-1}$.
Hence, $\FuncII$ should be given by integrating the product of $\FuncIII \ab (\DataValVec)$ and $\ab (J \ParamMLE \ab (\DataValVec))^{-1}$ on the ellipse $S_{\Param}$.
Here, since we consider the integral on an ellipse, which is locally one-dimensional, the integral defined on the two-dimensional Lebesgue measure does not work. 
Moreover, since an ellipse is not a line, we cannot apply the integral defined on the one-dimensional Lebesgue measure. Instead, we need to apply the integral defined on the one-dimensional Hausdorff measure.
}
\label{fig:coarea}
\end{figure*}

% \begin{figure*}
%     \centering
% \begin{tikzpicture}[scale=8.0]

% % 楕円を描く
% \draw[dotted] plot[domain=0:360, samples=100] ({sqrt(0.3)*cos(\x)}, {sqrt(0.3)/2*sin(\x)});

% % 微小円板を描く
% \foreach \x in {0, 10, ..., 360} {
%     \pgfmathsetmacro{\px}{sqrt(0.3)*cos(\x)};
%     \pgfmathsetmacro{\py}{sqrt(0.3)/2*sin(\x)};
%     \coordinate (P) at (\px, \py);
%     \pgfmathsetmacro{\tr}{veclen({2*sin(\x)},{cos(\x)})};
%     \coordinate (Q) at ({\px - 0.05*sin(\x)/\tr}, {\py + 0.025*cos(\x)/\tr});
%     \draw[dashed, gray] (P) circle[radius=0.05];
%     \draw[thick, black] (P) -- (Q);
    
%     % 楕円の接線と平行な線分を描く
%     % 楕円の接線の傾きは、楕円の方程式の微分から求める
%     % 楕円の方程式は (x/a)^2 + (y/b)^2 = 1
%     % 接線の傾きは dy/dx = - (a^2 * y) / (b^2 * x)
%     % \pgfmathsetmacro{\slope}{-(0.3 * sin(\x)) / (0.075 * cos(\x))} % a^2=0.3, b^2=0.075
%     % \pgfmathsetmacro{\angle}{atan(\slope)}
    
%     % % 接線と平行な線分を描く
%     % \draw[rotate around={\angle:(P)}] (P) -- ++(-0.05, 0) -- ++(0.1, 0);
% }

% \end{tikzpicture}
% \caption{Caption}
%     \label{fig:enter-label}
% \end{figure*}

\section{Key Idea: Coarea Formula}
\label{sec:Idea}
In this section, we describe the motivation and intuitive meaning of how the coarea formula, the main idea of this study, solves the problems described in the previous section in deriving the MC calculation formula for continuous cases. 
Recall that the difficulty in modifying discrete case discussion to one for continuous cases is Lemma \ref{lem:DiscreteGeneralDecomposition} cannot apply to continuous cases straightforwardly.
Specifically, \eqref{eqn:NaiveMapDecomposition} does not hold in general.
On the other hand, once we obtain a continuous version of Lemma \ref{lem:DiscreteGeneralDecomposition}, then we can also obtain the MC calculation formula for continuous cases based on Rissanen's direction \cite{Rissanen:2007:Information} \cite{Rissanen:2012:Optimal}
since the derivation there is more widely applicable and only depends on Lemma \ref{lem:DiscreteGeneralDecomposition}.
Hence, our motivation in this section is to obtain a continuous version of Lemma \ref{lem:DiscreteGeneralDecomposition}, to obtain the MC calculation formula for continuous cases in the next section.
Specifically, we are motivated to find the function $\FuncII$ such that it is given by the estimator map $\ParamMapVec$ and some integration on $\ParamMapVec^{-1} \ab(\ab\{\ParamVec\})$ and it satisfies
\begin{equation}
\label{eqn:GExpectation}
\int_{\DataSpace} \FuncIII \ab (\DataValVec) \odif{\Lebesgue^{\DataDim} \ab (\DataValVec)}
= 
\int_{\ParamSpace} \FuncII \ab (\ParamVec) \odif{\Lebesgue^{\ParamDim} \ab (\ParamVec)},
\end{equation}
for a given function $\FuncIII$.
In fact, the conclusion of this section is that the coarea formula, which is a fundamental result in geometric measurement theory (e.g., \cite{evans2018measure}), is the correct continuous version of Lemma \ref{lem:DiscreteGeneralDecomposition}, not \eqref{eqn:NaiveMapDecomposition}. 
To state the coarea formula, we will provide the mathematical preliminary, explaining what needs to be modified in \eqref{eqn:NaiveMapDecomposition} and the intuitive reasons for this.
As a result, this section includes the definition of a non-square version of the Jacobian determinant to precisely state the coarea formula, which is also a preparation for stating the main theorems of this study to calculate the MC.
\begin{remark}
\label{rem:GeneralEstimator}
In the following, our direction is to establish the continuous version of Theorem \ref{thm:YamanishiG}, which holds for any estimator $\ParamMapVec$.
Therefore, in the following, we discuss a general estimator $\ParamMapVec$, not limited to a MLE $\ParamMLEVec$.
Nevertheless, we can specialize the discussion into the MLE version if you substitute $\ParamMapVec = \ParamMLEVec$.
\end{remark}

Intuitively speaking, the equation \eqref{eqn:NaiveMapDecomposition}, which does not hold in general, has two significant issues that are essential when considering continuous space.
\begin{enumerate}
    \item While we want to evaluate the integral of the PDF on the set $\ParamMapVec^{-1} \ab(\ab\{\ParamVec\})$ in some sense, we cannot use the integral defined by the Lebesgue measure $\Lebesgue^{\DataDim}$ since the set $\ParamMapVec^{-1} \ab(\ab\{\ParamVec\})$ is dimension-deficient and its Lebesgue measure is zero in general.
    In other words, the inner integral's value in \eqref{eqn:NaiveMapDecomposition} is always zero as long as we apply the integration defined by the Lebesgue measure $\Lebesgue^{\DataDim}$.
    \item The ratio of the infinitesimal change of the data point $\DataValVec$ in the data space $\DataSpace$ to the associated infinitesimal change of the estimator $\ParamMapVec$ in the parameter space $\ParamSpace$ is not taken into account.
    Specifically, the ratio of the hypervolume of the infinitesimal parallelotope in the data space to the hypervolume of the infinitesimal parallelotope in the parameter space that corresponds to that parallelotope by the estimator is not considered.
\end{enumerate}
\begin{example}
Let us revisit Example \ref{exm:Ellipse}, which corresponds to Fig.~\ref{fig:coarea}.
Since $\ParamMLE^{-1} \ab(\ab\{\Param\})$ is an ellipse for $\Param \in \Real_{>0}$ as demonstrated by the case of $\Param = 0.2$ in Fig.~\ref{fig:coarea}, the integral $\int_{\ParamMLE^{-1} \ab(\ab\{\Param\})} \FuncIII \ab(\DataValVec) \odif{\Lebesgue^{2} \ab(\DataValVec)}$ is always zero regardless of what the function $\FuncIII$ is.
Hence, we need a suitable measure to define the integral on the ellipse.

Also, as demonstrated by the case of $\Param = 0.2$ and $\Delta \Param = 0.1$ in Fig.~\ref{fig:coarea}, the distance between the ellipse $\ParamMLE^{-1} \ab(\ab\{\Param\})$ and $\ParamMLE^{-1} \ab(\ab\{\Param + \Delta \Param\})$ is not uniform.
For $\FuncII$ to satisfy \eqref{eqn:GExpectation}, $\FuncII \ab(\Param) \Delta \Param$ should approximate the integral $\int_{\ParamMapVec^{-1} \ab([\Param, \Param + \Delta \Param])} \FuncIII \ab(\DataValVec) \odif{\Lebesgue^{2} \ab(\DataValVec)}$, where the integral domain corresponds to the gray area in Fig.~\ref{fig:coarea}.
Hence, we need to the distance between $\ParamMLE^{-1} \ab(\ab\{\Param\})$ and $\ParamMLE^{-1} \ab(\ab\{\Param + \Delta \Param\})$, which corresponds to the width of the gray area in Fig.~\ref{fig:coarea} to obtain the correct $\FuncII$.
\end{example}
In the following, we introduce the Hausdorff measure to solve the issue 1) and the non-square version of the Jacobian determinant to solve the issue 2).

\subsection{The Hausdorff measure to evaluate the integral on a dimension-deficient set}
As discussed above, the Lebesgue measure is not suitable for quantifying the ``size'' of a dimension-deficient set, such as $\ParamMapVec^{-1} \ab (\ab \{\ParamVec\})$, which plays a key role in Lemma \ref{lem:DiscreteGeneralDecomposition}.
Specifically, the Lebesgue measure of such a set is zero.
Hence, we need a new measure to quantify the ``size'' of such a set.
The Hausdorff measure meets our demand.
Intuitively, Hausdorff measure covers the set $\Set \subset \Real^{\DataDim}$ of interest by microballs and assigns to each microball the volume of the $s$-dimensional ball of the same diameter. Here, the diameter of a set $\Set \subset \Real^{\DataDim}$ is defined by $\Diam \ab (\Set) = \sup_{\boldsymbol{x}, \boldsymbol{y} \in \Set} \ab \|\boldsymbol{x}- \boldsymbol{y}\|_{2}$.
\begin{definition}[Hausdorff (outer) measure]
Let $V_{s} \ab (r)$ be the volume of the $s$-dimensional ball of radius $r$, defined by $V_{s} \ab (r) \DefEq \frac{\pi^{\frac{s}{2}}}{\Gamma \ab (s + \frac{1}{2})} r^{s}$.
Suppose $s \in \Real_{\ge0}$ and $\delta \in \overline{\Real}_{>0}$.
We define $\Hausdorff_{\delta}^{s}: 2^{\Real^{\DataDim}} \to \overline{\Real}_{\ge 0}$ by
\begin{equation}
\Hausdorff_{\delta}^{s} \ab (\Set)
\DefEq \inf \ab \{ \sum_{j}^{\infty} V_s \ab (\frac{\Diam \ab (C_j)}{2}) \middle| \begin{aligned}
&C_{1}, C_{2}, \dots \in \Real^{\DataDim}, \\ 
&\Set \subset \bigcup_{j}^{\infty} C_{j}, \\
&\Diam \ab (C_{j}) \le \delta. 
\end{aligned}\}.
\end{equation}
The \NewTerm{$s$-dimensional Hausdorff outer measure} $\ab (\Hausdorff^{s})^{*}: 2^{\Real^{\DataDim}} \to \overline{\Real}_{\ge 0}$ is define by
\begin{equation}
\ab (\Hausdorff^{s})^{*} \ab (\Set) = \lim_{\delta \searrow 0} \Hausdorff_{\delta}^{s} \ab (\Set).
\end{equation}
Define $\Sigma^{s, \DataDim} \subset 2^{\Real^{\DataDim}}$ by
\begin{equation}
\Sigma^{s, \DataDim}
\DefEq
\ab \{\SetI \in \Real^{\DataDim} \middle| \forall \SetII \in \Real^{\DataDim}, \ab (\Hausdorff^{s})^{*} \ab (\SetI \cap \SetII) + \ab (\Hausdorff^{s})^{*} \ab (\SetII \setminus \SetI) \}.
\end{equation}
Then, $\Hausdorff^{s} \DefEq \left. (\Hausdorff^{s})^{*}\right|_{\Sigma^{s, \DataDim}}: \Sigma^{s, \DataDim}
 \to \overline{\Real}_{\ge 0}$, which is given by restricting the domain of the outer measure $(\Hausdorff^{s})^{*}$ to $\Sigma^{s, \DataDim}$, is called the \NewTerm{$s$-dimensional Hausdorff measure}. 
\end{definition}

\begin{remark}
The only difference between the Hausdorff outer measure and the Hausdorff measure is in its domain of definition. Since this paper deals only with measurable functions, we can ignore those differences.
\end{remark}

By using the Hausdorff measure, as aimed, we can define an integral that is intrinsically meaningful on a dimension-deficient set, including $\ParamMapVec^{-1} \ab(\ab\{\ParamVec\})$.

\subsection{The non-square Jacobian determinant as a ratio of inifinitesimal changes in the data and parameter spaces}
Another issue is that the decomposition \eqref{eqn:NaiveDecomposition} does not consider the ratio between the infinitesimal change in parameter space and the infinitesimal in data space.
Recall that the ratio is correctly considered as the Jacobian determinant in the change of variables in integral by a bijective.
Similarly, we can consider the ratio by the non-square version of the Jacobian determinant.
\begin{definition}[the non-square version of the Jacobian determinant]
Let $\partial_{i}$ be the partial differential operator with respect to $i$-th element.
For $\FuncVec: \Real^{\DataDim} \to \Real^{\ParamDim}$, let $\nabla \FuncVec: \Real^{\DataDim} \to \Real^{\ParamDim, \DataDim}$ be the Jacobian matrix defined by 
\begin{equation}
\nabla \FuncVec \ab (\DataValVec)
= 
\bxmat[showleft=2,showtop=2,format=\partial_{#3} #1_{#2} \ab (\DataValVec)]
{\Func}{\ParamDim}{\DataDim}.
\end{equation}
Then, we define the non-square absolute Jacobian determinant $J \FuncVec: \Real^{\DataDim} \to \Real_{\ge 0}$ by
\begin{equation}
\begin{split}
J \FuncVec \ab (\DataValVec)
& \DefEq
\prod_{k = 1}^{\min\{\ParamDim, \DataDim\}}
\sigma_{k} \ab(\nabla \FuncVec)
\\
& =
\begin{cases}
\sqrt{\det \ab (\nabla \FuncVec \ab (\DataValVec)
\ab (\nabla \FuncVec \ab (\DataValVec)
)^\top)} & \quad \textrm{if $\DataDim \ge \ParamDim$,} \\
\sqrt{\det \ab (\ab (\nabla \FuncVec \ab (\DataValVec)
)^\top \nabla \FuncVec \ab (\DataValVec))} & \quad \textrm{if $\DataDim \le \ParamDim$,} \\
\end{cases}
\end{split}
\end{equation}
where $\sigma_{k} \ab(\nabla \FuncVec) \in \Real_{\ge 0}$ is the $k$-th singular value of $\nabla \FuncVec$.
\end{definition}
\begin{remark}
If $\DataDim=\ParamDim$, then $J \FuncVec \ab (\DataValVec)$ is nothing but the absolute Jacobian determinant $\ab |\det(\nabla \FuncVec \ab (\DataValVec))|$, which appears in the well-known integral formula for the change of variable by a bijective map.
\end{remark}
\begin{remark}[Geometric intuition of the Jacobian determinant]
Like square cases, the non-square Jacobian of a map $\FuncVec$ also indicates how an infinitesimal parallelotope is mapped by $\FuncVec$ and its determinant indicates how the volume is changed by the map.
To understand the meaning of the Jacobian determinant, the singular value decomposition of the Jacobian matrix $\nabla \FuncVec$ helps.
In the following, we assume $\DataDim \ge \ParamDim$, but a similar explanation holds for the cases when $\DataDim \le \ParamDim$.

First, the behavior of the function $\FuncVec$ in the neighborhood of a point $\DataValVec_{0} \in \Real^\DataDim$ is approximated as follows:
\begin{equation}
\FuncVec \ab(\DataValVec) - \FuncVec \ab(\DataValVec_{0})
\approx
\nabla \FuncVec \ab(\DataValVec_{0}) \ab(\DataValVec - \DataValVec_{0}).
\end{equation}
Let $\nabla \FuncVec \ab(\DataValVec_{0}) = \boldsymbol{U} \boldsymbol{\varSigma} \boldsymbol{V}^{\top}$ be the singular value decomposition of $\nabla \FuncVec \ab(\DataValVec_{0})$. 
Here, $\boldsymbol{U} = \ab[\xmat[showleft=2,format={#1_{#3}}]{\boldsymbol{u}}{1}{\ParamDim}] \in \Real^{\ParamDim, \ParamDim}$ and $\boldsymbol{V} = \ab[\xmat[showleft=2,format={#1_{#3}}]{\boldsymbol{v}}{1}{\DataDim}] \in \Real^{\DataDim, \DataDim}$ are orthogonal matrices, and 
\begin{equation}
\boldsymbol{\varSigma}
=
\begin{bmatrix}
\sigma_{1} & 0 & \cdots & 0 & 0 & \cdots & 0 \\
0 & \sigma_{2} & \cdots & 0 & 0 & \cdots & 0 \\
\vdots & \vdots & \ddots & \vdots & \vdots & \ddots & \vdots \\
0 & 0 & \cdots & \sigma_{\ParamDim} & 0 & \cdots & 0 \\
\end{bmatrix}
\in \Real^{\ParamDim, \DataDim}
\end{equation}
is a diagonal matrix, whose $k$-th diagonal entry $\sigma_{k} = \sigma_{k} \ab(\nabla \FuncVec \ab(\DataValVec_{0}))$ is the $k$-th singular value of $\nabla \FuncVec \ab(\DataValVec_{0})$.
At $\DataValVec_{0}$, the map $\FuncVec$ degenerates the vectors $\boldsymbol{v}_{\ParamDim + 1}, \boldsymbol{v}_{\ParamDim + 2}, \dots, \boldsymbol{v}_{\DataDim}$.
In other words, infinitesimal change of the point along with a vector in the subspace spanned by $\ab\{\boldsymbol{v}_{\ParamDim + 1}, \boldsymbol{v}_{\ParamDim + 2}, \dots, \boldsymbol{v}_{\DataDim}\}$ does not change the return value of $\FuncVec$.
Let us consider the infinitesimal $\ParamDim$-dimensional parallelotope formed by $\ParamDim$ normal vectors $\lambda \boldsymbol{v}_{1}, \lambda \boldsymbol{v}_{2}, \dots, \lambda \boldsymbol{v}_{\DataDim}$ with length $\lambda$ at $\DataValVec_{0}$ in $\Real^{\DataDim}$.
This parallelotope's volume is $\lambda^{\ParamDim}$.
The map $\FuncVec$ maps the $k$-th infinitesimal vector $\lambda \boldsymbol{v}_{k} \in \Real^{\DataDim}$ at $\DataValVec_{0}$ to $\lambda \sigma_{k} \boldsymbol{u}_{k} \in \Real^{\ParamDim}$.
Hence, the image of the original $\ParamDim$-dimensional parallelotope at $\DataValVec_{0}$ in $\Real^{\DataDim}$ is the $\ParamDim$-dimensional parallelotope at $\FuncVec \ab(\DataValVec_{0})$ in $\Real^{\ParamDim}$ formed by $\ParamDim$ infinitesimal vectors $\lambda \sigma_{1} \boldsymbol{v}_{1}, \lambda \sigma_{2} \boldsymbol{v}_{2}, \dots, \lambda \sigma_{\DataDim} \boldsymbol{v}_{\DataDim}$.
Hence, the volume of the image parallelotope is $\prod_{k=1}^{\ParamDim} \lambda \sigma_{k} = J \FuncVec \ab(\DataValVec_{0}) \lambda^{\ParamDim}$.
That is, the map $\FuncVec$ multiplies the volume of an infinitesimal parallelotope by $J \FuncVec \ab(\DataValVec_{0})$ times. 
\end{remark}

As explained in the above remark, the map $\FuncVec$ multiplies the volume of an infinitesimal parallelotope at $\DataValVec_{0}$ by $J \FuncVec \ab(\DataValVec_{0})$ times.
Conversely, $\FuncVec$ associates a parallelotope at $\FuncVec \ab(\DataValVec_{0})$ in $\Real^{\ParamDim}$ with one at $\DataValVec_{0}$ in $\Real^{\DataDim}$, whose volume is $\frac{1}{J \FuncVec \ab(\DataValVec_{0})}$ times as large as the original parallelotope.
Hence, roughly speaking, when we attempt to evaluate an integral in $\Real^{\DataDim}$ through the integral in $\Real^{\ParamDim}$ associated by the map $\FuncVec$, we should multiply the integrated by $\frac{1}{J \FuncVec \ab(\DataValVec_{0})}$. This idea is formulated as the coarea formula in the next subsection.

\subsection{Coarea formula}
Now, we have defined the Hausdorff measure and the non-square version of the Jacobian determinant.
From the discussion in the previous subsections, we would expect that we could correct the equation \eqref{eqn:NaiveMapDecomposition}, which does not hold in general, with modifications by replacing the Lebesgue measure with the Hausdorff measure in the inner integral and dividing it by the non-square version of the Jacobian determinant of the estimator map $\ParamMapVec$. 
And this expectation is indeed correct. 
This is exactly what the coarea formula states.
\begin{theorem}[coarea formula: integral version]
\label{thm:CoareaChangeVariable}
Suppose $\DataDim \ge \ParamDim$.
Let $\FuncVec: \Real^{\DataDim} \to \Real^{\ParamDim}$ be Lipschitz continuous,
and assume that $\essinf_{\DataValVec} J \FuncVec \ab (\DataValVec) > 0$, i.e., $\Lebesgue^{\DataDim} \ab (\ab\{\DataValVec \middle| \textrm{$J \FuncVec \ab (\DataValVec)$ is undefined or $J \FuncVec \ab (\DataValVec) = 0$} \}) = 0$.
Let $\FuncIII: \Real^{\DataDim} \to \Real$ be a \textbf{nonnegative} measurable function.
Then, we have that
\begin{equation}
\label{eqn:CoareaIntegral}
\begin{split}
& \int_{\Set} \FuncIII (\DataValVec) \odif{\Lebesgue^{\DataDim} \ab (\DataValVec)}
\\
& =
\int_{\Real^{\ParamDim}}
\ab (
\int_{\FuncVec^{-1} \ab (\ab \{\ParamVec\})} \FuncIII (\DataValVec) \ab(J \FuncVec \ab (\DataValVec))^{-1'} \odif{\Hausdorff^{\DataDim-\ParamDim} \ab (\DataValVec)}
)
\odif{\Lebesgue^{\ParamDim} \ab (\ParamVec)},
\end{split}
\end{equation}
where $\ab(J \FuncVec \ab (\DataValVec))^{-1'}$ is defined by
\begin{equation}
\ab(J \FuncVec \ab (\DataValVec))^{-1'}
\DefEq
\frac{1 \ab\{J \FuncVec \ab (\DataValVec) > 0\}}{J \FuncVec \ab (\DataValVec)}
=
\begin{cases}
\frac{1}{J \FuncVec \ab (\DataValVec)} \quad & \textrm{if $J \FuncVec \ab (\DataValVec) > 0$}, \\
0 \quad & \textrm{otherwise}. \\
\end{cases}
\end{equation}
Here, \eqref{eqn:CoareaIntegral} also claims that if either side of them is positive infinity, then the other side is also positive infinity.
\end{theorem}
\begin{remark}
The specific version of the integral formula provided by Theorem \ref{thm:CoareaChangeVariable} is not standard, but we can prove it from the standard coarea formula to convert the measure as in Appendix.
Similar versions are well-known (e.g., Theorem 3.11. in \cite{evans2018measure}), where not the nonnegativity but the integrability of the integrated is assumed.
In this paper, we only handle nonnegative functions as integrands, so we do not extend the theorem to general cases.
\end{remark}
\begin{remark}
\label{rem:Differentiablility}
Theorem \ref{thm:CoareaChangeVariable} assumes that the integrand $\FuncVec$ is Lipschitz continuous, but not differentiable everywhere. 
Hence, we cannot define $J \FuncVec$ everywhere.
Nevertheless, $\FuncVec$ is differentiable \Emph{almost} everywhere according to Rademacher's theorem (e.g., Theorem 3.2. in \cite{evans2018measure}), and thus we can define $J \FuncVec$ \Emph{almost} everywhere.
This is an intuitive reason why the right-hand side of \eqref{eqn:CoareaIntegral} is well-defined.
\end{remark}
In the next section, we state our main theorems, using the definitions of the Hausdorff measure and the non-square version of the Jacobian determinant, and the coarea formula.

\section{Main Results}
\label{sec:Main}
In this section, we state our main theorems.
Recall that we aim to show that we can calculate the MC using the PDF of the estimator's distribution.
To achieve that, we first provide the specific form of the PDF of the distribution, then we provide the formula to calculate the MC using the PDF. 

First, we provide a specific formula of the PDF of the probability measure $\ParamMapVec_{\sharp} \ProbabilityMeasure_{\ParamVec}$, i.e., the PDF of the estimator $\ParamMapVec \ab (\DataRVVec_{\ParamVec})$ where $\DataRVVec_{\ParamVec}$ follows the distribution $\ProbabilityMeasure_{\ParamVec}$.
\begin{theorem}[Estimator's PDF]
\label{thm:EstimatorPDF}
Suppose $\DataSpace \subset \Real^{\DataDim}$ and $\ParamSpace \subset \Real^{\ParamDim}$.
Let $\ab \{\ProbabilityMeasure_{\ParamVec}\}_{\ParamVec \in \ParamSpace}$ be a continuous PPM on the data space $\DataSpace$ and let $\PDF \ab[\ProbabilityMeasure_{\ParamVec}]: \DataSpace \to \Real_{\ge 0}$ be the (canonical) PDF of $\ProbabilityMeasure_{\ParamVec}$ for $\ParamVec \in \ParamSpace$.
Let $\ParamMapVec: \DataSpace \to \ParamSpace$ be Lipschitz continuous and assume that $\essinf_{\DataValVec} J \ParamMapVec \ab (\DataValVec) > 0$, i.e., $\Lebesgue^{\DataDim} \ab (J \ParamMapVec \ab (\DataValVec) = 0 \textrm{ or undefined}) = 0$.
Define $\PDF \ab[\ParamMapVec_{\sharp} \ProbabilityMeasure_{\ParamVec}]: \ParamSpace \to \Real_{\ge 0}$ by
\begin{equation}
\PDF \ab[\ParamMapVec_{\sharp} \ProbabilityMeasure_{\ParamVec}] \ab (\ParamVec')
=
\int_{\ParamMapVec^{-1} \ab (\ab \{\ParamVec'\})} \PDF \ab[\ProbabilityMeasure_{\ParamVec}] (\DataValVec) \ab(J \ParamMapVec \ab (\DataValVec))^{-1'}
\odif{\Hausdorff^{\DataDim-\ParamDim} \ab (\DataValVec)},
\end{equation}
where $\ab(J \ParamMapVec \ab (\DataValVec))^{-1'}$ is defined by
\begin{equation}
\ab(J \ParamMapVec \ab (\DataValVec))^{-1'}
\DefEq
\frac{1 \ab\{J \ParamMapVec \ab (\DataValVec) > 0\}}{J \ParamMapVec \ab (\DataValVec)}
=
\begin{cases}
\frac{1}{J \ParamMapVec \ab (\DataValVec)} \quad & \textrm{if $J \FuncVec \ab (\DataValVec) > 0$}, \\
0 \quad & \textrm{otherwise}. \\
\end{cases}
\end{equation}
Then, $\PDF \ab[\ParamMapVec_{\sharp} \ProbabilityMeasure_{\ParamVec}]$ is a PDF of $\ParamMapVec_{\sharp} \ProbabilityMeasure_{\ParamVec}$, i.e., for any measurable set $\Set \subset \ParamSpace$, the following holds:
\begin{equation}
\ab (\ParamMapVec_{\sharp} \ProbabilityMeasure_{\ParamVec})
\ab (\Set)
=
\int_{A} \PDF \ab[\ParamMapVec_{\sharp} \ProbabilityMeasure_{\ParamVec}] \ab (\ParamVec') \odif{\Lebesgue^{\ParamDim} \ab (\ParamVec')}.
\end{equation}
In other words, if random variable $\DataRVVec_{\ParamVec}$'s distribution is $\ProbabilityMeasure_{\ParamVec}$, then the following holds:
\begin{equation}
\Probability \ab (\ParamMapVec
(\DataRVVec_{\ParamVec}) \in \Set)
=
\int_{A} \PDF \ab[\ParamMapVec_{\sharp} \ProbabilityMeasure_{\ParamVec}] \ab (\ParamVec') \odif{\Lebesgue^{\ParamDim} \ab (\ParamVec')}.
\end{equation}
\end{theorem}
\begin{remark}
Regarding the differentiability of $\ParamMapVec$, refer to Remark \ref{rem:Differentiablility}, where $\FuncVec$ plays the same role as $\ParamMapVec$ here.
\end{remark}

\begin{proof}[Proof of Theorem \ref{thm:EstimatorPDF}]
\begin{equation}
\ab (\ParamMapVec_{\sharp} \ProbabilityMeasure_{\ParamVec}) \ab (\Set)
=
\int_{\DataSpace} 1 \ab \{\ParamMapVec \ab (\DataValVec) \in \Set\} \PDF \ab[\ProbabilityMeasure_{\ParamVec}] \ab (\DataValVec) \odif{\Lebesgue^{\DataDim} \ab (\DataValVec)}.
\end{equation}
According to Theorem \ref{thm:CoareaChangeVariable}, the right-hand side equals to
\begin{equation}
\small
\int_{\ParamMapVec \ab (\DataSpace)} \ab (\int_{\ParamMapVec^{-1} \ab (\ab \{\ParamVec'\})} 1 \ab \{\ParamMapVec \ab (\DataValVec) \in \Set\} \PDF \ab[\ProbabilityMeasure_{\ParamVec}] \ab (\DataValVec) \ab(J \ParamMapVec \ab (\DataValVec))^{-1'} \odif{\Hausdorff^{\DataDim-\ParamDim} \ab (\DataValVec)}) \odif{\Lebesgue^{\ParamDim} \ab (\ParamVec)}.
\end{equation}
In the inner integral, $\DataValVec \in \ParamMapVec^{-1} \ab (\ab \{\ParamVec'\})$ holds, which is equivalent to $\ParamMapVec \ab (\DataValVec) = \ParamVec'$. 
Therefore, the above iterated integral equals
\begin{equation}
\small
\begin{split}
& \int_{\ParamMapVec \ab (\DataSpace)} \ab (\int_{\ParamMapVec^{-1} \ab (\ab \{\ParamVec'\})} 1 \ab \{\ParamVec' \in \Set\} \PDF \ab[\ProbabilityMeasure_{\ParamVec}] \ab (\DataValVec) \ab(J \ParamMapVec \ab (\DataValVec))^{-1'} \odif{\Hausdorff^{\DataDim-\ParamDim} \ab (\DataValVec)}) \odif{\Lebesgue^{\ParamDim} \ab (\ParamVec)} \\
& =
\int_{\ParamMapVec \ab (\DataSpace)} 1 \ab \{\ParamVec' \in \Set\} \ab (\int_{\ParamMapVec^{-1} \ab (\ab \{\ParamVec'\})} \PDF \ab[\ProbabilityMeasure_{\ParamVec}] \ab (\DataValVec)\ab(J \ParamMapVec \ab (\DataValVec))^{-1'} \odif{\Hausdorff^{\DataDim-\ParamDim} \ab (\DataValVec)}) \odif{\Lebesgue^{\ParamDim} \ab (\ParamVec)} \\
& =
\int_{\Set} \ab (\int_{\ParamMapVec^{-1} \ab (\ab \{\ParamVec'\})} \PDF \ab[\ProbabilityMeasure_{\ParamVec}] \ab (\DataValVec)\ab(J \ParamMapVec \ab (\DataValVec))^{-1'} \odif{\Hausdorff^{\DataDim-\ParamDim} \ab (\DataValVec)}) \odif{\Lebesgue^{\ParamDim} \ab (\ParamVec)} \\
& =
\int_{\Set} \ab
\PDF \ab[\ParamMapVec_{\sharp} \ProbabilityMeasure_{\ParamVec}] \ab (\ParamVec')
\odif{\Lebesgue^{\ParamDim} \ab (\ParamVec)},
\end{split}
\end{equation}
which completes the proof.
\end{proof}

Now, we have obtained the PDF of the estimator $\ParamMapVec$.
We are ready to state our final theorem asserting that we can calculate the MC using the PDF given by Theorem \ref{thm:EstimatorPDF}.

\begin{theorem}[MC Calculation formula for a continuous PPM]
\label{thm:ContinuousMC}
Under the assumptions in Theorem \ref{thm:EstimatorPDF} and the definition of $\PDF \ab[\ParamMapVec_{\sharp} \ProbabilityMeasure_{\ParamVec}]$,
we can calculate the LMC as follows:
\begin{equation}
\begin{split}
\ExponentialComplexity_{\Luckiness} \ab[\ProbabilityMeasure_{\ParamSpace}]
& \DefEq
\int_{\DataSpace} \PDF \ab[\ProbabilityMeasure_{\ParamMapVec \ab (\DataValVec)}] \ab (\DataValVec) \Luckiness \ab(\ParamMapVec \ab (\DataValVec)) \odif{\Lebesgue^{\DataDim} \ab (\DataValVec)}
\\
& = 
\int_{\ParamSpace} \PDF \ab[\ParamMapVec_{\sharp} \ProbabilityMeasure_{\ParamVec}] \ab (\ParamVec) \Luckiness (\ParamVec) \odif{\Lebesgue^{\ParamDim} \ab (\ParamVec)}.
\end{split}
\end{equation}
In particular, we can calculate the MC by setting $\Luckiness = 1_{\ParamSpace}$ as follows:
\begin{equation}
\begin{split}
\ExponentialComplexity \ab[\ProbabilityMeasure_{\ParamSpace}]
& \DefEq
\int_{\DataSpace} \PDF \ab[\ProbabilityMeasure_{\ParamMapVec \ab (\DataValVec)}] \ab (\DataValVec) \odif{\Lebesgue^{\DataDim} \ab (\DataValVec)}
\\
& = 
\int_{\ParamSpace} \PDF \ab[\ParamMapVec_{\sharp} \ProbabilityMeasure_{\ParamVec}] \ab (\ParamVec) \odif{\Lebesgue^{\ParamDim} \ab (\ParamVec)}.
\end{split}
\end{equation}
\end{theorem}

\begin{example}[The LMC of the exponential distribution]
\label{exm:ExponentialLMC}
Consider the $\NData$ product measure of the exponential distribution model, whose PDF $\PDF \ab[\ab(\ProbabilityMeasure_{\theta})^{\NData}]: \ab(\Real_{\ge 0})^{\NData} \to \Real_{\ge 0}$ is given by
\begin{equation}
\PDF \ab[\ab(\ProbabilityMeasure_{\theta})^{\NData}] \ab(\DataVal_{1}^{\NData})
=
\prod_{\IDatum=1}^{\NData} \PDF \ab[\ProbabilityMeasure_{\theta}] \ab(\DataVal_{\IDatum})
=
\frac{1}{\theta} \exp \ab(- \frac{\sum_{\IDatum=1}^{\NData} \DataVal_{\IDatum}}{\theta})
.
\end{equation}
It is known that the MLE $\hat{\uptheta}: \Real_{\ge 0}^{\NData} \to \Real_{\ge 0}$ with respect to $\theta$ is given by $\hat{\uptheta} \ab(\DataVal_{1}^{\NData}) = \frac{1}{\NData} \sum_{\IDatum=1}^{\NData} \DataVal_{\IDatum}$.
From the reproducibility of the Gamma distribution, we can see that the PDF of $\hat{\uptheta}_{\sharp} \ab(\ProbabilityMeasure_{\theta})^{\NData}$ equals that of the Gamma distribution with shape parameter $\NData$ and scale parameter $\frac{\theta}{\NData}$. 
Hence, we have the following:
\begin{equation}
\PDF \ab[\hat{\uptheta}_{\sharp} \ab(\ProbabilityMeasure_{\theta})^{\NData}] \ab(\theta')
=
\frac{1}{\Gamma \ab(\NData)} \ab(\frac{\NData}{\theta})^{\NData} \ab(\theta')^{\NData - 1} \exp(-\frac{\NData}{\theta} \theta').
\end{equation}
When $\theta = \theta'$, we have the following:
\begin{equation}
\PDF \ab[\hat{\uptheta}_{\sharp} \ab(\ProbabilityMeasure_{\theta})^{\NData}] \ab(\theta)
=
\frac{\NData^{\NData}}{\Gamma \ab(\NData) \theta} \exp(-\NData).
\end{equation}
Therefore, according to Theorem \ref{thm:ContinuousMC}, we can calculate the LMC as follows:
\begin{equation}
\begin{split}
\ExponentialComplexity_{\Luckiness} \ab[\ProbabilityMeasure_{\ParamSpace}]
& = 
\int_{\ParamSpace} \PDF \ab[\ParamMapVec_{\sharp} \ProbabilityMeasure_{\ParamVec}] \ab (\ParamVec) \Luckiness (\ParamVec) \odif{\Lebesgue^{\ParamDim} \ab (\ParamVec)}
\\
& = 
\int_{\Real_{>0}} \frac{\NData^{\NData}}{\Gamma \ab(\NData) \theta} \exp(-\NData) \Luckiness (\theta) \odif{\Lebesgue^{\ParamDim} \ab (\theta)} \\
& = 
\frac{\NData^{\NData}}{\Gamma \ab(\NData)} \exp(-\NData) \int_{\Real_{>0}} \frac{1}{\theta} \Luckiness (\theta) \odif{\Lebesgue^{\ParamDim} \ab (\theta)}.
\end{split}
\end{equation}
If the luckiness $\Luckiness$ is given by the indicator function $\Luckiness = 1_{[\theta_{\min}, \theta_{\max}]}$, then we have
\begin{equation}
\ExponentialComplexity_{\Luckiness} \ab[\ProbabilityMeasure_{\ParamSpace}]
=
\frac{\NData^{\NData}}{\Gamma \ab(\NData)} \exp(-\NData) \ab(\log \theta_{\max} - \log \theta_{\min}).
\end{equation}

This calculation itself has been given in, e.g., \cite{Rissanen:2007:Information} \cite{Rissanen:2012:Optimal} \cite{yamanishi2023learning}.
Our contribution lies in the mathematical proofs behind the formula. 
\end{example}

\begin{remark}
\begin{enumerate}
    \item As also explained in Remark \ref{rem:GeneralEstimator}, Theorems \ref{thm:EstimatorPDF} and \ref{thm:ContinuousMC} apply to any estimator as long as it satisfies the assumptions in the Theorems. In particular, it does not have to be the MLE. All we need to obtain the result for the MLE $\ParamMLEVec$ is to substitute $\ParamMapVec = \ParamMLEVec$.
    \item In the right-hand side of Theorem \ref{thm:ContinuousMC}, $\ParamVec$, which is the variable of integral, also appears as the index of $\ProbabilityMeasure$. In other words, the index of $\ProbabilityMeasure$ also varies while integrating.
    \item Theorem \ref{thm:ContinuousMC} is the first theorem with a correct proof that shows we can calculate the exact value of the MC by integrating a function given by the PDF of the estimator.
\end{enumerate}
\end{remark}
\begin{remark}[Limitation of Theorem \ref{thm:ContinuousMC} and future work]
Theorem \ref{thm:ContinuousMC} states that the equation \eqref{eqn:ContinuousGFormula} holds if the PDF of the estimator $\ParamMapVec$ is chosen as in Theorem \ref{thm:EstimatorPDF}.
This achieves this study's goal essentially, which was to prove the formula \eqref{eqn:ContinuousGFormula}.

Nevertheless, we still have room for improving Theorem \ref{thm:ContinuousMC}, because, strictly speaking, the PDF is not unique for values on a set of zero measures. 
When calculating probabilities for a fixed probability distribution, differences in values on the set of zero measures are negligible and therefore not a problem. 
However, in the case of Theorem \ref{thm:ContinuousMC}, the parameter of the probability distribution also changes during integration. 
Hence, the lack of uniqueness of PDF can change the integration result. 
Thus, strictly speaking, it is an open question whether Theorem \ref{thm:ContinuousMC} can be applied if we take an arbitrary PDF of $\ParamMapVec_{\sharp} \ProbabilityMeasure_{\ParamVec}$.

In reality, PDF of one distribution calculated by multiple methods rarely differ, so the above concern is not considered a practical problem.
Nevertheless, mathematically it is an interesting future work. 
In particular, since a continuous PDF is unique if it exists, if the PDF given in Theorem 8 is continuous, the calculation method of Theorem 9 can be used even if a continuous PDF of $\ParamMapVec_{\sharp} \ProbabilityMeasure_{\ParamVec}$ is calculated in another way. 
Therefore, one important future work is to find the conditions under which the PDF given in Theorem 8 is continuous.
\end{remark}
\begin{proof}[Proof of Theorem \ref{thm:ContinuousMC}]
Applying Theorem \ref{thm:CoareaChangeVariable}, we have that
\begin{equation}
\begin{split}
& \int_{\DataSpace} \PDF \ab[\ProbabilityMeasure_{\ParamMapVec \ab (\DataValVec)}] \ab (\DataValVec) \Luckiness \ab(\ParamMapVec \ab (\DataValVec)) \odif{\Lebesgue^{\DataDim} \ab (\DataValVec)}
\\ 
& = \int_{\ParamMapVec \ab (\DataSpace)} \ab (\int_{\ParamMapVec^{-1} \ab (\ab \{\ParamVec\})} \frac{\PDF \ab[\ProbabilityMeasure_{\ParamMapVec \ab (\DataValVec)}] \ab (\DataValVec) \Luckiness \ab(\ParamMapVec \ab (\DataValVec))}{J \ParamMapVec \ab (\DataValVec)} \odif{\Hausdorff^{\DataDim-\ParamDim} \ab (\DataValVec)}) \odif{\Lebesgue^{\ParamDim} \ab (\ParamVec)}.
\end{split}
\end{equation}
Here, in the internal integral, $\ParamMapVec^{-1} \ab (\ab \{\ParamVec\})$ holds, which is equivalent to $\ParamMapVec \ab (\DataValVec) = \ParamVec$.
Therefore, the above iterated integral equals the following:
\begin{equation}
\int_{\ParamMapVec \ab (\DataSpace)} \ab (\int_{\ParamMapVec^{-1} \ab (\ab \{\ParamVec\})} \frac{\PDF \ab[\ProbabilityMeasure_{\ParamVec}] \ab (\DataValVec) \Luckiness \ab(\ParamVec)}{J \ParamMapVec \ab (\DataValVec)} \odif{\Hausdorff^{\DataDim-\ParamDim} \ab (\DataValVec)}) \odif{\Lebesgue^{\ParamDim} \ab (\ParamVec)}.
\end{equation}
Here, the inner integral satisfies
\begin{equation}
\int_{\ParamMapVec^{-1} \ab (\ab \{\ParamVec\})} \frac{\PDF \ab[\ProbabilityMeasure_{\ParamVec}] \ab (\DataValVec) \Luckiness \ab(\ParamVec)}{J \ParamMapVec \ab (\DataValVec)} \odif{\Hausdorff^{\DataDim-\ParamDim} \ab (\DataValVec)}
=
\PDF \ab[\ParamMapVec_{\sharp} \ProbabilityMeasure_{\ParamVec}] \ab (\ParamVec) \Luckiness \ab(\ParamVec)
\end{equation}
by the definition of $\PDF \ab[\ParamMapVec_{\sharp} \ProbabilityMeasure_{\ParamVec}]$ given in Theorem \ref{thm:EstimatorPDF}.
It completes the proof.
\end{proof}

\section{Conclusion}
This paper has, for the first time, justified the MC calculation formula using the MLE density, which has widely been used in the past, for continuous PPM cases. Specifically, we have provided a specific representation of the PDF of general estimators, including MLE, where the data follows a continuous probability distribution.
Using the PDF of the estimator given above, we have derived the MC calculation formula with complete proof. 
The MC calculation formula for which we provided complete proof has widely been applied due to its simplicity, our mathematical justification of the formula leads to justification of those applications. 
In this sense, the contribution of this study is significant.

% or


\appendices

\section{Proof of Theorem \ref{thm:CoareaChangeVariable}}
In this section, we prove Theorem \ref{thm:CoareaChangeVariable} using the following standard version of the coarea formula converting a measure.
\begin{theorem}[coarea formula: measure version (e.g., Theorem 3.10 in \cite{evans2018measure})]
\label{thm:CoareaMeasure}
Suppose $\DataDim \ge \ParamDim$.
Let $\FuncVec: \Real^{\DataDim} \to \Real^{\ParamDim}$ be Lipschitz continuous.
Then for each $\Lebesgue^{\DataDim}$-measurable set $\Set \subset \Real^{\DataDim}$,
\begin{equation}
\int_{\Set} J \FuncVec \ab (\DataValVec) \odif{\Lebesgue^{\DataDim} \ab (\DataValVec)}
=
\int_{\Real^{\ParamDim}}
\Hausdorff^{\DataDim-\ParamDim} (\Set \cap \FuncVec^{-1} \ab (\ab \{\ParamVec\})) \odif{\Lebesgue^{\ParamDim} \ab (\ParamVec)}.
\end{equation}
\end{theorem}
Applying Theorem \ref{thm:CoareaMeasure}, we can derive the following Theorem.
\begin{theorem}
Suppose $\DataDim \ge \ParamDim$.
Let $\FuncVec: \Real^{\DataDim} \to \Real^{\ParamDim}$ be Lipschitz continuous.
Let $\FuncII: \Real^{\DataDim} \to \Real$ be a \textbf{nonnegative} measurable function.
Then, we have that
\begin{equation}
\label{eqn:CoareaIntegralOriginal}
\begin{split}
& \int_{\Real^{\ParamDim}} \FuncII (\DataValVec) J \FuncVec \ab (\DataValVec) \odif{\Lebesgue^{\DataDim} \ab (\DataValVec)}
\\
& =
\int_{\Real^{\ParamDim}}
\ab (
\int_{\FuncVec^{-1} \ab (\ab \{\ParamVec\})} \FuncII (\DataValVec)
\odif{\Hausdorff^{\DataDim-\ParamDim} \ab (\DataValVec)}
)
\odif{\Lebesgue^{\ParamDim} \ab (\ParamVec)}.
\end{split}
\end{equation}
Note that the above equation also claims that if either side is positive infinity then the other side is also positive infinity.
\end{theorem}
\begin{proof}
Since $\FuncII$ is a nonnegative measurable function, there exists an infinite sequence functions $\mathrm{s}_{1}, \mathrm{s}_{2}, \dots: \Real^{\DataDim} \to \overline{\Real}$ such that 
\begin{itemize}
    \item Each $\mathrm{s}_{n}$ is a simple function, i.e., there exist a nonnegative integer $M_{n} \in \Integer_{\ge 0}$, $M_{n}$ nonnegative real numbers $\alpha_{1, n}, \alpha_{2, n}, \dots, \alpha_{M_{n}, n} \in \Real_{\ge 0}$, and $M_{n}$ subsets $A_{1, n}, A_{2, n}, \dots, A_{M_n, n} \in \Real_{\ge 0}$ such that
    \begin{equation}
        \label{eqn:SimpleDef}
        \mathrm{s}_{n} \ab(\DataValVec) = \sum_{m=1}^{M_n} \alpha_{m,n} 1_{A_{m,n}} \ab(\DataValVec),
    \end{equation}
    for all $\DataValVec \in \Real^{\DataDim}$.
    \item The sequence $\mathrm{s}_{1}, \mathrm{s}_{2}, \dots$ is monotonous nondecreasing, i.e., for any $\DataValVec \in \Real^{\DataDim}$ and $n \in \Integer_{>0}$, the following holds:
    \begin{equation}
        \mathrm{s}_{n} \ab(\DataValVec) \le \mathrm{s}_{n+1} \ab(\DataValVec).
    \end{equation}
    \item The sequence $\mathrm{s}_{1}, \mathrm{s}_{2}, \dots$ converges to $\FuncII$ pointwisely, i.e., for any $\DataValVec \in \Real^{\DataDim}$, the following holds:
    \begin{equation}
        \lim_{n \to +\infty} \mathrm{s}_{n} \ab(\DataValVec) = \mathrm{\FuncII} \ab(\DataValVec).
    \end{equation}
\end{itemize}
Now, let us start evaluating the left-hand side of the Theorem.
We have the following equation:
\begin{equation}
\begin{split}
\int_{\Real^{\DataDim}} \FuncII \ab(\DataValVec) J \FuncVec \ab(\DataValVec) \odif{\Lebesgue^{\DataDim} \ab (\DataValVec)}
& =
\int_{\Real^{\DataDim}} \lim_{n \to +\infty} \mathrm{s}_{n} \ab(\DataValVec) J \FuncVec \ab(\DataValVec) \odif{\Lebesgue^{\DataDim} \ab (\DataValVec)}
\\
& =
\lim_{n \to +\infty} \int_{\Real^{\DataDim}} \mathrm{s}_{n} \ab(\DataValVec) J \FuncVec \ab(\DataValVec) \odif{\Lebesgue^{\DataDim} \ab (\DataValVec)},
\end{split}
\end{equation}
where the second equation holds because the nonnegativity of $J \FuncVec$ allows us to apply the monotonous convergence theorem.
Note that both sides of the equation may be $+\infty$. 
The equation also claims that if either side is $+\infty$ then the other is also $+\infty$.

By \eqref{eqn:SimpleDef} and the linearity of the integral, we have that
\begin{equation}
\begin{split}
& \lim_{n \to +\infty} \int_{\Real^{\DataDim}} \mathrm{s}_{n} \ab(\DataValVec) J \FuncVec \ab(\DataValVec) \odif{\Lebesgue^{\DataDim} \ab (\DataValVec)}
\\
& =
\lim_{n \to +\infty} \int_{\Real^{\DataDim}} \sum_{m=1}^{M_n} \alpha_{m,n} 1_{A_{m,n}} \ab(\DataValVec) J \FuncVec \ab(\DataValVec) \odif{\Lebesgue^{\DataDim} \ab (\DataValVec)}
\\
& =
\lim_{n \to +\infty} \sum_{m=1}^{M_n} \alpha_{m,n} \int_{\Real^{\DataDim}}  1_{A_{m,n}} \ab(\DataValVec) J \FuncVec \ab(\DataValVec) \odif{\Lebesgue^{\DataDim} \ab (\DataValVec)}
\\
& =
\lim_{n \to +\infty} \sum_{m=1}^{M_n} \alpha_{m,n} \int_{A_{m,n}} J \FuncVec \ab(\DataValVec) \odif{\Lebesgue^{\DataDim} \ab (\DataValVec)}.
\end{split}    
\end{equation}
Now, we are ready to apply Theorem \ref{thm:CoareaMeasure}, to obtain the following equation.
\begin{equation}
\label{eqn:CoareaApplied}
\begin{split}
& \lim_{n \to +\infty} \sum_{m=1}^{M_n} \alpha_{m,n} \int_{A_{m,n}} J \FuncVec \ab(\DataValVec) \odif{\Lebesgue^{\DataDim} \ab (\DataValVec)}
\\
&=
\lim_{n \to +\infty} \sum_{m=1}^{M_n} \alpha_{m,n} \int_{\Real^{\ParamDim}}
\Hausdorff^{\DataDim-\ParamDim} (A_{m,n} \cap \FuncVec^{-1} \ab (\ab \{\ParamVec\})) \odif{\Lebesgue^{\ParamDim} \ab (\ParamVec)}.
\end{split}
\end{equation}
We can evaluate the integral of the right-hand side as follows:
\begin{equation}
\begin{split}
& \int_{\Real^{\ParamDim}} \Hausdorff^{\DataDim-\ParamDim} (A_{m,n} \cap \FuncVec^{-1} \ab (\ab \{\ParamVec\})) \odif{\Lebesgue^{\ParamDim} \ab (\ParamVec)}
\\
& =
\int_{\Real^{\ParamDim}} \ab(\int_{\Real^{\DataDim}}  1_{A_{m,n} \cap \FuncVec^{-1} \ab (\ab \{\ParamVec\})} \ab(\DataValVec) \odif{\Hausdorff^{\DataDim-\ParamDim} \ab(\DataValVec)}) \odif{\Lebesgue^{\ParamDim} \ab(\ParamVec)}
\\
& =
\int_{\Real^{\ParamDim}} \ab(\int_{\Real^{\DataDim}}  1_{\FuncVec^{-1} \ab (\ab \{\ParamVec\})} \ab(\DataValVec) 1_{A_{m,n}} \ab(\DataValVec) \odif{\Hausdorff^{\DataDim-\ParamDim} \ab(\DataValVec)}) \odif{\Lebesgue^{\ParamDim} \ab(\ParamVec)}
\\
& =
\int_{\Real^{\ParamDim}} \ab(\int_{\FuncVec^{-1} \ab (\ab \{\ParamVec\})} 1_{A_{m,n}} \ab(\DataValVec) \odif{\Hausdorff^{\DataDim-\ParamDim} \ab(\DataValVec)}) \odif{\Lebesgue^{\ParamDim} \ab(\ParamVec)},
\\
\end{split}
\end{equation}
where the first equation is due to the definition of the integral of an indicator function and the third equation follows from the definition of limiting the domain of integration.
Using the above equation, we can evaluate the right-hand side of \eqref{eqn:CoareaApplied} as follows:
\begin{equation}
\begin{split}
& \lim_{n \to +\infty} \sum_{m=1}^{M_n} \alpha_{m,n} \int_{\Real^{\ParamDim}}
\Hausdorff^{\DataDim-\ParamDim} (A_{m,n} \cap \FuncVec^{-1} \ab (\ab \{\ParamVec\})) \odif{\Lebesgue^{\ParamDim} \ab (\ParamVec)}
\\
& =
\lim_{n \to +\infty} \sum_{m=1}^{M_n} \alpha_{m,n} \int_{\Real^{\ParamDim}} \ab(\int_{\FuncVec^{-1} \ab (\ab \{\ParamVec\})} 1_{A_{m,n}} \ab(\DataValVec) \odif{\Hausdorff^{\DataDim-\ParamDim} \ab(\DataValVec)}) \odif{\Lebesgue^{\ParamDim} \ab(\ParamVec)}
\\
& =
\lim_{n \to +\infty} \int_{\Real^{\ParamDim}} \ab(\int_{\FuncVec^{-1} \ab (\ab \{\ParamVec\})} \sum_{m=1}^{M_n} \alpha_{m,n} 1_{A_{m,n}} \ab(\DataValVec) \odif{\Hausdorff^{\DataDim-\ParamDim} \ab(\DataValVec)}) \odif{\Lebesgue^{\ParamDim} \ab(\ParamVec)}
\\
& =
\lim_{n \to +\infty} \int_{\Real^{\ParamDim}} \ab(\int_{\FuncVec^{-1} \ab (\ab \{\ParamVec\})} \mathrm{s}_{n} \ab(\DataValVec) \odif{\Hausdorff^{\DataDim-\ParamDim} \ab(\DataValVec)}) \odif{\Lebesgue^{\ParamDim} \ab(\ParamVec)}.
\end{split}
\end{equation}
Apply the monotone convergence theorem, then we obtain
\begin{equation}
\begin{split}
& \lim_{n \to +\infty} \int_{\Real^{\ParamDim}} \ab(\int_{\FuncVec^{-1} \ab (\ab \{\ParamVec\})} \mathrm{s}_{n} \ab(\DataValVec) \odif{\Hausdorff^{\DataDim-\ParamDim} \ab(\DataValVec)}) \odif{\Lebesgue^{\ParamDim} \ab(\ParamVec)}
\\
& =
\int_{\Real^{\ParamDim}} \ab(\int_{\FuncVec^{-1} \ab (\ab \{\ParamVec\})} \lim_{n \to +\infty} \mathrm{s}_{n} \ab(\DataValVec) \odif{\Hausdorff^{\DataDim-\ParamDim} \ab(\DataValVec)}) \odif{\Lebesgue^{\ParamDim} \ab(\ParamVec)}.
\\
& =
\int_{\Real^{\ParamDim}} \ab(\int_{\FuncVec^{-1} \ab (\ab \{\ParamVec\})} \FuncII \ab(\DataValVec) \odif{\Hausdorff^{\DataDim-\ParamDim} \ab(\DataValVec)}) \odif{\Lebesgue^{\ParamDim} \ab(\ParamVec)},
\\
\end{split}
\end{equation}
which completes the proof.
\end{proof}

\begin{proof}[Proof of Theorem \ref{thm:CoareaChangeVariable}]
Substitute $\FuncII \ab(\DataValVec) = \FuncIII \ab(\DataValVec) \ab(J \FuncVec \ab(\DataValVec))^{-1'} = \FuncIII \ab(\DataValVec) \frac{1\ab\{J \FuncVec \ab(\DataValVec) > 0\}}{J \FuncVec \ab(\DataValVec)}$ in \eqref{eqn:CoareaIntegralOriginal}.
Then, the left-hand side equals 
\begin{equation}
\begin{split}
\int_{\Real^{\ParamDim}} \FuncII (\DataValVec) J \FuncVec \ab (\DataValVec) \odif{\Lebesgue^{\DataDim} \ab (\DataValVec)}
& =
\int_{\Real^{\ParamDim}} \FuncIII \ab(\DataValVec) \frac{1\ab\{J \FuncVec \ab(\DataValVec) > 0\}}{J \FuncVec \ab(\DataValVec)} J \FuncVec \ab (\DataValVec) \odif{\Lebesgue^{\DataDim} \ab (\DataValVec)}
\\
& =
\int_{\Real^{\ParamDim}} \FuncIII \ab(\DataValVec) 1\ab\{J \FuncVec \ab(\DataValVec) > 0\} \odif{\Lebesgue^{\DataDim} \ab (\DataValVec)}
\\
& =
\int_{\Real^{\ParamDim}} \FuncIII \ab(\DataValVec)  \odif{\Lebesgue^{\DataDim} \ab (\DataValVec)},
\\
\end{split}
\end{equation}
where the last equation holds since $J \FuncVec \ab(\DataValVec) > 0$ is satisfied almost everywhere with respect to $\Lebesgue^{\DataDim}$.
\begin{equation}
\begin{split}
& \int_{\Real^{\ParamDim}}
\ab (
\int_{\FuncVec^{-1} \ab (\ab \{\ParamVec\})} \FuncII (\DataValVec)
\odif{\Hausdorff^{\DataDim-\ParamDim} \ab (\DataValVec)}
)
\odif{\Lebesgue^{\ParamDim} \ab (\ParamVec)}
\\
& =
\int_{\Real^{\ParamDim}}
\ab (
\int_{\FuncVec^{-1} \ab (\ab \{\ParamVec\})} \FuncIII \ab(\DataValVec) \frac{1\ab\{J \FuncVec \ab(\DataValVec) > 0\}}{J \FuncVec \ab(\DataValVec)}
\odif{\Hausdorff^{\DataDim-\ParamDim} \ab (\DataValVec)}
)
\odif{\Lebesgue^{\ParamDim} \ab (\ParamVec)},
\end{split}
\end{equation}
which completes the proof.
\end{proof}

\section*{Acknowledgment}

This work is partially supported by JSPS KAKENHI  JP24H00703.


\ifCLASSOPTIONcaptionsoff
  \newpage
\fi

\bibliographystyle{IEEEtran}
\bibliography{ref.bib}

\begin{thebibliography}{10}
\providecommand{\url}[1]{#1}
\csname url@samestyle\endcsname
\providecommand{\newblock}{\relax}
\providecommand{\bibinfo}[2]{#2}
\providecommand{\BIBentrySTDinterwordspacing}{\spaceskip=0pt\relax}
\providecommand{\BIBentryALTinterwordstretchfactor}{4}
\providecommand{\BIBentryALTinterwordspacing}{\spaceskip=\fontdimen2\font plus
\BIBentryALTinterwordstretchfactor\fontdimen3\font minus \fontdimen4\font\relax}
\providecommand{\BIBforeignlanguage}[2]{{%
\expandafter\ifx\csname l@#1\endcsname\relax
\typeout{** WARNING: IEEEtran.bst: No hyphenation pattern has been}%
\typeout{** loaded for the language `#1'. Using the pattern for}%
\typeout{** the default language instead.}%
\else
\language=\csname l@#1\endcsname
\fi
#2}}
\providecommand{\BIBdecl}{\relax}
\BIBdecl

\bibitem{Grunwald:2007}
P.~D. Gr{\"u}nwald, \emph{The minimum description length principle}.\hskip 1em plus 0.5em minus 0.4em\relax MIT press, 2007.

\bibitem{Shtarkov:1987}
Y.~M. Shtar'kov, ``Universal sequential coding of single messages,'' \emph{Problemy Peredachi Informatsii}, vol.~23, no.~3, pp. 3--17, 1987.

\bibitem{Rissanen:1989}
J.~Rissanen, \emph{Stochastic complexity in statistical inquiry}.\hskip 1em plus 0.5em minus 0.4em\relax World Scientific, 1989.

\bibitem{Rissanen:1996}
------, ``Fisher information and stochastic complexity,'' \emph{IEEE Transactions on Information Theory,}, vol.~42, no.~1, pp. 40--47, 1996.

\bibitem{tabus2003classification}
I.~Tabus, J.~Rissanen, and J.~Astola, ``Classification and feature gene selection using the normalized maximum likelihood model for discrete regression,'' \emph{Signal Processing}, vol.~83, no.~4, pp. 713--727, 2003.

\bibitem{Myung+:2006}
J.~I. Myung, D.~J. Navarro, and M.~A. Pitt, ``Model selection by normalized maximum likelihood,'' \emph{Journal of Mathematical Psychology}, vol.~50, no.~2, pp. 167--179, 2006.

\bibitem{Roos+:2008}
T.~Roos, T.~Silander, P.~Kontkanen, and P.~Myllymaki, ``Bayesian network structure learning using factorized nml universal models,'' in \emph{Information Theory and Applications Workshop, 2008}.\hskip 1em plus 0.5em minus 0.4em\relax IEEE, 2008, pp. 272--276.

\bibitem{schmidt2010estimating}
D.~F. Schmidt and E.~Makalic, ``Estimating the order of an autoregressive model using normalized maximum likelihood,'' \emph{IEEE Transactions on Signal Processing}, vol.~59, no.~2, pp. 479--487, 2010.

\bibitem{giurcuaneanu2011variable}
C.~D. Giurc{\u{a}}neanu, S.~A. Razavi, and A.~Liski, ``Variable selection in linear regression: Several approaches based on normalized maximum likelihood,'' \emph{Signal Processing}, vol.~91, no.~8, pp. 1671--1692, 2011.

\bibitem{zhang2012model}
J.~Zhang, ``Model selection with informative normalized maximum likelihood: Data prior and model prior,'' in \emph{Descriptive and normative approaches to human behavior}.\hskip 1em plus 0.5em minus 0.4em\relax World Scientific, 2012, pp. 303--319.

\bibitem{Hirai&Yamanishi:2013:Efficient}
S.~Hirai and K.~Yamanishi, ``Efficient computation of normalized maximum likelihood codes for gaussian mixture models with its applications to clustering,'' \emph{IEEE Transactions on Information Theory}, vol.~59, no.~11, pp. 7718--7727, 2013.

\bibitem{staniczenko2014selecting}
P.~P. Staniczenko, M.~J. Smith, and S.~Allesina, ``Selecting food web models using normalized maximum likelihood,'' \emph{Methods in Ecology and Evolution}, vol.~5, no.~6, pp. 551--562, 2014.

\bibitem{Ito:2016}
Y.~Ito, S.~Oeda, and K.~Yamanishi, ``Selecting ranks and detecting their changes for non-negative matrix factorization using normalized maximum likelihood coding,'' in \emph{Proceedings of SIAM International Conference on Data Mining}, 2016.

\bibitem{Suzuki:2016}
A.~Suzuki, K.~Miyaguchi, and K.~Yamanishi, ``Structure selection for convolutive non-negative matrix factorization using normalized maximum likelihood coding,'' in \emph{The IEEE International Conference on Data Mining series}.\hskip 1em plus 0.5em minus 0.4em\relax IEEE, 2016.

\bibitem{yamanishi2019decomposed}
K.~Yamanishi, T.~Wu, S.~Sugawara, and M.~Okada, ``The decomposed normalized maximum likelihood code-length criterion for selecting hierarchical latent variable models,'' \emph{Data Mining and Knowledge Discovery}, vol.~33, pp. 1017--1058, 2019.

\bibitem{kellen2020selecting}
D.~Kellen and K.~C. Klauer, ``Selecting amongst multinomial models: An apologia for normalized maximum likelihood,'' \emph{Journal of Mathematical Psychology}, vol.~97, p. 102367, 2020.

\bibitem{yamanishi2023detecting}
K.~Yamanishi and S.~Hirai, ``Detecting signs of model change with continuous model selection based on descriptive dimensionality,'' \emph{Applied Intelligence}, vol.~53, no.~22, pp. 26\,454--26\,471, 2023.

\bibitem{grunwald2019minimum}
P.~Gr{\"u}nwald and T.~Roos, ``Minimum description length revisited,'' \emph{International journal of mathematics for industry}, vol.~11, no.~01, p. 1930001, 2019.

\bibitem{Kraft:1949}
L.~G. Kraft, ``A device for quantizing, grouping, and coding amplitude-modulated pulses,'' Ph.D. dissertation, Massachusetts Institute of Technology, 1949.

\bibitem{McMillan:1956}
B.~McMillan, ``Two inequalities implied by unique decipherability,'' \emph{IRE Transactions on Information Theory}, vol.~2, no.~4, pp. 115--116, 1956.

\bibitem{Kontkanen&Myllymaki:2007:Linear}
P.~Kontkanen and P.~Myllym{\"a}ki, ``A linear-time algorithm for computing the multinomial stochastic complexity,'' \emph{Information Processing Letters}, vol. 103, no.~6, pp. 227--233, 2007.

\bibitem{Mononen&Myllymaki:2007:Fast}
T.~Mononen and P.~Myllym{\"a}ki, ``Fast nml computation for naive bayes models,'' in \emph{International Conference on Discovery Science}.\hskip 1em plus 0.5em minus 0.4em\relax Springer, 2007, pp. 151--160.

\bibitem{barron1998minimum}
A.~Barron, J.~Rissanen, and B.~Yu, ``The minimum description length principle in coding and modeling,'' \emph{IEEE transactions on information theory}, vol.~44, no.~6, pp. 2743--2760, 1998.

\bibitem{Hirai+:2013}
S.~Hirai and K.~Yamanishi, ``Efficient computation of normalized maximum likelihood codes for gaussian mixture models with its applications to clustering,'' \emph{IEEE Transactions on Information Theory}, vol.~59, no.~11, pp. 7718--7727, 2013.

\bibitem{yamanishi2023learning}
K.~Yamanishi, \emph{Learning with the Minimum Description Length Principle}.\hskip 1em plus 0.5em minus 0.4em\relax Springer, 2023.

\bibitem{Rissanen:2007:Information}
J.~Rissanen, \emph{Information and complexity in statistical modeling}.\hskip 1em plus 0.5em minus 0.4em\relax Springer Science \& Business Media, 2007.

\bibitem{Rissanen:2012:Optimal}
------, \emph{Optimal Estmation of Parameters}.\hskip 1em plus 0.5em minus 0.4em\relax Cambridge, 2012.

\bibitem{Rissanen:2000}
------, ``Mdl denoising,'' \emph{IEEE Transactions on Information Theory}, vol.~46, no.~7, pp. 2537--2543, 2000.

\bibitem{Suzuki&Yamanishi:2018:Exact}
A.~Suzuki and K.~Yamanishi, ``Exact calculation of normalized maximum likelihood code length using fourier analysis,'' in \emph{2018 IEEE International Symposium on Information Theory (ISIT)}.\hskip 1em plus 0.5em minus 0.4em\relax IEEE, 2018, pp. 1211--1215.

\bibitem{suzuki2021fourier}
------, ``Fourier-analysis-based form of normalized maximum likelihood: Exact formula and relation to complex bayesian prior,'' \emph{IEEE Transactions on Information Theory}, vol.~67, no.~9, pp. 6164--6178, 2021.

\bibitem{evans2018measure}
L.~Evans, \emph{Measure theory and fine properties of functions}.\hskip 1em plus 0.5em minus 0.4em\relax Routledge, 2018.

\end{thebibliography}

\end{document}